\documentclass[11pt,a4paper]{amsart}
\usepackage{amssymb,amsthm}
\usepackage{bm}

\usepackage[truedimen,margin=30truemm]{geometry}

\usepackage[pdftex,colorlinks=true,bookmarks=true,citecolor=blue,linkcolor=blue,pdfauthor={},pdftitle={}]{hyperref}

\theoremstyle{plain}
\newtheorem{theorem}{Theorem}[section]
\newtheorem{lemma}[theorem]{Lemma}
\newtheorem{proposition}[theorem]{Proposition}

\theoremstyle{definition}
\newtheorem{definition}[theorem]{Definition}
\newtheorem{assumption}[theorem]{Assumption}

\theoremstyle{remark}
\newtheorem{remark}[theorem]{Remark}

\numberwithin{equation}{section}

\DeclareMathOperator*{\esup}{ess\,sup}

\begin{document}

\title[Parabolic $p$-Laplace equation in a moving domain]{Weak solutions to the parabolic $p$-Laplace equation in a moving domain under a Neumann type boundary condition}

\author[T.-H. Miura]{Tatsu-Hiko Miura}
\address{Graduate School of Science and Technology, Hirosaki University, 3, Bunkyo-cho, Hirosaki-shi, Aomori, 036-8561, Japan}
\email{thmiura623@hirosaki-u.ac.jp}

\subjclass[2020]{35K20, 35K92, 35R37}

\keywords{Parabolic $p$-Laplace equation, moving domain, weak solution}

\begin{abstract}
  This paper studies the parabolic $p$-Laplace equation with $p>2$ in a moving domain under a Neumann type boundary condition corresponding to the total mass conservation.
  We establish the existence and uniqueness of a weak solution by the Galerkin method in evolving Bochner spaces and a monotonicity argument.
  The main difficulty is in characterizing the weak limit of the nonlinear gradient term, where we need to deal with a term which comes from the boundary condition and cannot be absorbed into a monotone operator.
  To overcome this difficulty, we prove a uniform-in-time Friedrichs type inequality on a moving domain with time-dependent basis functions and make use of it to get the strong convergence of approximate solutions.
  We also show that the time derivative exists in the $L^2$ sense when given data have a better regularity, and discuss extension of the existence and uniqueness results to a Leray--Lions type operator.
\end{abstract}

\maketitle

\section{Introduction} \label{S:Intro}
Let $T\in(0,\infty)$ and $n\geq1$.
For $t\in[0,T]$, let $\Omega_t$ be a bounded domain in $\mathbb{R}^n$ that moves in time.
We assume that $\Omega_t$ is $C^2$ in space and $C^1$ in time (see Assumption \ref{A:Domain} for the precise assumption).
For $p\in(2,\infty)$, we consider the parabolic $p$-Laplace equation
\begin{align} \label{E:pLap_MoDo}
  \left\{
  \begin{alignedat}{3}
    \partial_tu-\mathrm{div}(|\nabla u|^{p-2}\nabla u) &= f &\quad &\text{in} &\quad &Q_T := \textstyle\bigcup_{t\in(0,T)}\Omega_t\times\{t\}, \\
    |\nabla u|^{p-2}\partial_\nu u+V_\Omega u &= 0 &\quad &\text{on} &\quad &\partial_\ell Q_T := \textstyle\bigcup_{t\in(0,T)}\partial\Omega_t\times\{t\}, \\
    u|_{t=0} &= u_0 &\quad &\text{in} &\quad &\Omega_0
  \end{alignedat}
  \right.
\end{align}
for given functions $f$ and $u_0$.
Here, $\partial_\nu$ is the outer normal derivative on the boundary $\partial\Omega_t$ and $V_\Omega$ is the scalar outer normal velocity of $\partial\Omega_t$.
We impose a Neumann type (or Robin) boundary condition in order to get the total mass conservation
\begin{align*}
  \frac{d}{dt}\int_{\Omega_t}u\,dx = \int_{\Omega_t}f\,dx, \quad t\in(0,T)
\end{align*}
by the Reynolds transport theorem and the divergence theorem.

We consider weak solutions to \eqref{E:pLap_MoDo} in the framework of evolving Bochner spaces and a weak material time derivative introduced in \cite{AlElSt15_PM,AlCaDjEl23}.
The purpose of this paper is to establish the existence and uniqueness of a weak solution for the given data
\begin{align*}
  f \in L_{[W^{1,p}]^\ast}^{p'} = L^{p'}(0,T;[W^{1,p}(\Omega_t)]^\ast), \quad u_0\in L^2(\Omega_0),
\end{align*}
where $1/p+1/p'=1$ (see Theorem \ref{T:Exi_Uni}).
We refer to Section \ref{SS:Pre_FS} for the precise definition of evolving Bochner spaces $L_{\mathcal{X}}^q$.
In a weak formulation of \eqref{E:pLap_MoDo} (see Definition \ref{D:WS_pLap}), we use a weak material derivative (a weak time derivative along a velocity of $\Omega_t$) by following the abstract framework of \cite{AlElSt15_PM,AlCaDjEl23} (see Definition \ref{D:We_MatDe}).
When
\begin{align*}
  f \in L_{L^2}^2 = L^2(0,T;L^2(\Omega_t)), \quad u_0\in W^{1,p}(\Omega_0),
\end{align*}
we also show in Theorem \ref{T:Reg_dt} that the weak material derivative (and thus the time derivative) of a weak solution actually exists in the $L^2$ sense.
The existence and uniqueness results in Theorem \ref{T:Exi_Uni} can be extended to a parabolic equation with a Leray--Lions type operator that generalizes the $p$-Laplacian.
We discuss it in Section \ref{S:LerLio}.

Partial differential equations (PDEs) in time-dependent domains appear in various fields such as biology, engineering, and fluid mechanics (see \cite{KnoKre15} and the reference cited therein).
There is an extensive literature on PDEs in moving domains, especially on linear parabolic equations (see e.g. \cite{GiaSav96,BrHuLi97,Sav97,BonGua01,BuChSy04}).
Evolution equations with monotone operators such as the parabolic $p$-Laplace equation have been studied well when a domain does not move in time (see e.g. \cite{LaSoUr68,Lio69,Zei90_2B}).
Also, several authors have studied evolution equations with monotone operators in the case of space-time noncylindrical domains and the Dirichlet boundary condition (see e.g. \cite{Par13,CaNoOr17,AlCaDjEl23,Nak25}).
In this paper, we consider the Neumann type boundary condition in \eqref{E:pLap_MoDo} that is different from the Dirichlet one and physically relevant in view of mass conservation.
We give a new idea for dealing with an additional difficulty coming from the Neumann type boundary condition.

Let us briefly explain the outline of the proofs and the main difficulty.
First, we show the uniqueness of a weak solution based on a monotonicity argument.
Usually, we compute the time derivative of the $L^2$-norm of the difference $v$ of two weak solutions by using its weak form.
In the present case, however, this procedure generates the term
\begin{align*}
  \int_0^T\bigl(v(t),[\mathbf{v}_\Omega\cdot\nabla v](t)\bigr)_{L^2(\Omega_t)}\,dt,
\end{align*}
which cannot be properly absorbed into the integrals of $|v|^2$ and $|\nabla v|^p$ due to $p>2$.
Here, $\mathbf{v}_\Omega$ is a velocity of $\Omega_t$ (see Section \ref{SS:Pre_MoDo}) and this problematic term comes from the boundary condition of \eqref{E:pLap_MoDo} (see Section \ref{S:WF_Uni} for the derivation of a weak form of \eqref{E:pLap_MoDo}).
To circumvent this difficulty, we compute the time derivative of the (mollified) $L^1$-norm of $v$ by following the idea of the proof of \cite[Theorem 2.9]{CaNoOr17} (see Proposition \ref{P:Uni} for details).

To prove the existence of a weak solution, we carry out the Galerkin method with time-dependent basis functions $w_k^t$ on $\Omega_t$, which are advected by the velocity $\mathbf{v}_\Omega$ from smooth basis functions $w_k^0$ on $\Omega_0$ that form an orthonormal basis of $L^2(\Omega_0)$.
Then, as usual, we construct finite dimensional approximate solutions $u_N$, derive energy estimates for $u_N$ that give the weak convergence of $u_N$, and show that the weak limit of $u_N$ is indeed a weak solution to \eqref{E:pLap_MoDo} by a monotonicity argument (see Section \ref{S:Exis} for details).
Also, we get a regularity of the time derivative of a weak solution by the Galerkin method with higher order energy estimates (see Section \ref{S:Reg} for details).

The main difficulty is that we need to deal with the limits as $N\to\infty$ of
\begin{align} \label{E:Intro_NL}
  \int_0^T\bigl(u_N(t),[u_N\,\mathrm{div}\,\mathbf{v}_\Omega](t)\bigr)_{L^2(\Omega_t)}\,dt, \quad \int_0^T\bigl(u_N(t),[\mathbf{v}_\Omega\cdot\nabla u_N](t)\bigr)_{L^2(\Omega_t)}\,dt
\end{align}
when we characterize the weak limit of $|\nabla u_N|^{p-2}\nabla u_N$ by a monotonicity argument.
Again, these terms come from the boundary condition of \eqref{E:pLap_MoDo} (see Section \ref{S:WF_Uni}) and do not appear when a domain does not move in time.
The first term can be absorbed into a monotone operator by considering $e^{-\gamma t}u_N(t)$ with a large constant $\gamma>0$ as in \cite[Proposition 7.7]{AlCaDjEl23}.
However, this idea is not applicable to the second term.
Thus, instead of dealing with $e^{-\gamma t}u_N(t)$, we prove the strong convergence of $u_N$ in $L_{L^2}^2$.
We follow the idea used in the case of non-moving domains \cite[Chapter V, Theorem 6.7]{LaSoUr68}, but the novelty here is that we derive and make use of the following uniform-in-time Friedrichs type inequality
\begin{align*}
  \|\psi\|_{L^2(\Omega_t)}^2 \leq c\left(\sum_{k=1}^{K_\varepsilon}|(\psi,w_k^t)_{L^2(\Omega_t)}|^2+\varepsilon\|\psi\|_{H^1(\Omega_t)}^2\right)
\end{align*}
for all $\varepsilon>0$, $t\in[0,T]$, and $\psi\in H^1(\Omega_t)$ with constants $c>0$ and $K_\varepsilon\in\mathbb{N}$ independent of $t$ and $\psi$ (see Lemma \ref{L:Fried}).
Note that this is not trivial, since the inequality is uniform in $t$ and the family $\{w_k^t\}_k$ is not orthonormal in $L^2(\Omega_t)$ for $t>0$ in general while it is so at $t=0$.
We will get this inequality by using the Friedrichs inequality at $t=0$ with the orthonormal basis $\{w_k^0\}_k$ of $L^2(\Omega_0)$ and a suitable change of variables.
Here, we also note that the strong convergence of $u_N$ enables us not only to get the convergence of \eqref{E:Intro_NL} but also to avoid dealing with the time traces $u_N(0)$ and $u_N(T)$ that usually appear when one uses a monotonicity argument (see e.g. \cite{LaSoUr68,Lio69,Zei90_2B,AlCaDjEl23}).

Lastly, let us give comments on the case where $\Omega_t$ is a moving thin domain.
When $p=2$, we studied the heat equation in a moving thin domain around a moving closed hypersurface in \cite{Miu17,Miu23_HMTD}.
We investigated the behavior of a solution as the thickness of a thin domain tends to zero, derived the heat equation on a moving surface as a limit problem, and compared solutions to the original and limit problems in terms of the thickness of a thin domain.
Recently, in \cite{Miu26pre_pLMTD}, we studied the parabolic $p$-Laplace equation \eqref{E:pLap_MoDo} in a moving thin domain and derived a limit problem on a moving surface.
The proof of that result is based on the results and ideas presented in this paper.

The rest of this paper is organized as follows.
In Section \ref{S:Prelim}, we fix notations on moving domains and function spaces, and prepare some basic results.
We give the definition of a weak solution to \eqref{E:pLap_MoDo} and prove its uniqueness in Section \ref{S:WF_Uni}.
Also, we construct a weak solution by the Galerkin method in Section \ref{S:Exis}.
In Section \ref{S:Reg}, we get the regularity of the time derivative of a weak solution.
In Section \ref{S:LerLio}, we briefly observe that the existence and uniqueness results can be extended to a Leray--Lions type operator.
We make concluding remarks in Section \ref{S:Concl}.

\section{Preliminaries} \label{S:Prelim}

We fix several notations and give basic results used in the following sections.

In what follows, we write $c$ for a general positive constant independent of variables, indices, and functions appearing in inequalities.
Also, we consider vectors in $\mathbb{R}^n$ as column vectors.
When $\Psi$ is a mapping from a subset of $\mathbb{R}^n$ into $\mathbb{R}^n$, we write
\begin{align*}
  \nabla\Psi =
  \begin{pmatrix}
    \partial_1\Psi & \cdots & \partial_n\Psi
  \end{pmatrix}^{\mathrm{T}},
  \quad \partial_1\Psi,\dots,\partial_n\Psi\in\mathbb{R}^n, \quad \partial_i = \frac{\partial}{\partial x_i}
\end{align*}
for the gradient matrix of $\Psi$.
Here, we take the transpose to write
\begin{align*}
  \nabla(u\circ\Psi) = \nabla\Psi[\nabla u\circ\Psi]
\end{align*}
for a function $u$ on $\mathbb{R}^n$.
For a Banach space $\mathcal{X}$, we denote by $[\mathcal{X}]^\ast$ the dual space of $\mathcal{X}$ and by $\langle\cdot,\cdot\rangle_{\mathcal{X}}$ the duality product between $[\mathcal{X}]^\ast$ and $\mathcal{X}$.
Also, we abuse the notation
\begin{align*}
  (u_1,u_2)_{L^2(\Omega_t)} = \int_{\Omega_t}u_1(x)u_2(x)\,dx
\end{align*}
whenever the right-hand integral makes sense.

\subsection{Moving domain} \label{SS:Pre_MoDo}
Fix a finite $T>0$.
For $t\in[0,T]$, let $\Omega_t$ be a bounded domain in $\mathbb{R}^n$ with $n\geq1$, and let $Q_T$ and $\overline{Q_T}$ be a space-time domain and its closure given by
\begin{align*}
  Q_T := \bigcup_{t\in(0,T)}\Omega_t\times\{t\}, \quad \overline{Q_T} := \bigcup_{t\in[0,T]}\overline{\Omega_t}\times\{t\}.
\end{align*}
We make the following assumption on the regularity and motion of $\Omega_t$.

\begin{assumption} \label{A:Domain}
  We assume that $\Omega_0$ is a $C^2$ domain and there exists a mapping
  \begin{align*}
    \overline{\Omega_0}\times[0,T]\ni(X,t)\mapsto \Phi_t(X)\in\mathbb{R}^n, \quad \Phi_{(\cdot)}\in C^1([0,T];C^2(\overline{\Omega_0})^n)
  \end{align*}
  such that $\Phi_0$ is the identity mapping on $\overline{\Omega_0}$ and, for each $t\in[0,T]$,
  \begin{itemize}
    \item $\Phi_t\colon\overline{\Omega_0}\to\overline{\Omega_t}$ is a $C^2$-diffeomorphism with $\Phi_t(\Omega_0)=\Omega_t$ and $\Phi_t(\partial\Omega_0)=\partial\Omega_t$,
    \item $\Phi_t|_{\Omega_0}\colon\Omega_0\to\Omega_t$ and $\Phi_t|_{\partial\Omega_0}\colon\partial\Omega_0\to\partial\Omega_t$ are also $C^2$-diffeomorphisms.
  \end{itemize}
\end{assumption}

Note that $\Omega_t$ is also a $C^2$ domain for all $t\in[0,T]$ by Assumption \ref{A:Domain}.
For each $t\in[0,T]$, we denote by $\Phi_t^{-1}$ the inverse mapping of $\Phi_t$.
The regularity of $\Phi_{(\cdot)}^{-1}$ with respect to time is not stated in Assumption \ref{A:Domain}, but we have (at least) the next result.
Note that we write $\nabla\Phi_t^{-1}$ for the gradient matrix of $\Phi_t^{-1}$ in $x$ for each $t\in[0,T]$.

\begin{lemma} \label{L:Pinv_Reg}
  The mappings $\Phi_{(\cdot)}^{-1}$ and $\nabla\Phi_{(\cdot)}^{-1}$ are of class $C^1$ on $\overline{Q_T}$.
\end{lemma}

\begin{proof}
  Let $\Theta(X,t):=(\Phi_t(X),t)$ for $(X,t)\in\overline{\Omega_0}\times[0,T]$.
  Then, by Assumption \ref{A:Domain},
  \begin{align*}
    \Theta\colon\overline{\Omega_0}\times[0,T]\to\overline{Q_T}, \quad \Theta\in C^1(\overline{\Omega_0}\times[0,T])^{n+1},
  \end{align*}
  and the inverse mapping is $\Theta^{-1}(x,t)=(\Phi_t^{-1}(x),t)$.
  Moreover, the gradient matrix of $\Theta$ in $X$ and $t$ is invertible on $\overline{\Omega_0}\times[0,T]$, since it is of the form
  \begin{align*}
    \nabla_{X,t}\Theta =
    \begin{pmatrix}
      \nabla\Phi_t & \mathbf{0} \\
      [\partial_t\Phi_t]^{\mathrm{T}} & 1
    \end{pmatrix},
    \quad \mathbf{0} = (0,\dots,0)^\mathrm{T} \in \mathbb{R}^n, \quad \partial_t\Phi_t\in\mathbb{R}^n,
  \end{align*}
  and $\nabla\Phi_t$ is invertible by Assumption \ref{A:Domain}.
  Thus, we can apply the inverse mapping theorem to $\Theta$ to find that $\Theta^{-1}$ and thus $\Phi_{(\cdot)}^{-1}$ are of class $C^1$ on $\overline{Q_T}$.
  Also, since
  \begin{align*}
    \nabla\Phi_t^{-1}(x) = [\nabla\Phi_t(\Phi_t^{-1}(x))]^{-1}, \quad (x,t) \in \overline{Q_T},
  \end{align*}
  we see that $\nabla\Phi_{(\cdot)}^{-1}$ is of class $C^1$ on $\overline{Q_T}$ by the regularity of $\nabla\Phi_{(\cdot)}$ and $\Phi_{(\cdot)}^{-1}$.
\end{proof}

Let $J_t:=\det\nabla\Phi_t$ on $\overline{\Omega_0}$ for $t\in[0,T]$.
By Assumption \ref{A:Domain}, $J_{(\cdot)}$ is of class $C^1$ and does not vanish on $\overline{\Omega_0}\times[0,T]$.
Moreover, $J_0\equiv1>0$ since $\Phi_0$ is the identity mapping.
Thus,
\begin{align} \label{E:Det_Bd}
  c_0 \leq J_t(X) \leq c_1, \quad |\nabla J_t(X)| \leq c_1 \quad\text{for all}\quad (X,t)\in\overline{\Omega_0}\times[0,T]
\end{align}
with some constants $c_0,c_1>0$.
In particular,
\begin{align} \label{E:Volume}
  c^{-1}|\Omega_0| \leq |\Omega_t| \leq c|\Omega_0| \quad\text{for all}\quad t\in[0,T]
\end{align}
for the volume $|\Omega_t|:=\int_{\Omega_t}\,dx=\int_{\Omega_0}J_t(X)\,dX$ of $\Omega_t$.
Also,
\begin{align} \label{E:Grad_Bd}
  c_2^{-1}|\nabla u| \leq |(\nabla U)\circ\Phi_t^{-1}| \leq c_2|\nabla u| \quad\text{on}\quad \Omega_t
\end{align}
for all $t\in[0,T]$, $u\colon\Omega_t\to\mathbb{R}$, and $U:=u\circ\Phi_t$ on $\Omega_0$ with some constant $c_2>0$.

We define the velocity field $\mathbf{v}_\Omega\colon\overline{Q_T}\to\mathbb{R}^n$ associated with the mapping $\Phi_{(\cdot)}$ by
\begin{align*}
  \mathbf{v}_\Omega(x,t) := \partial_t\Phi_t(\Phi_t^{-1}(x)) = \frac{\partial}{\partial t}\Bigl(\Phi_t(X)\Bigr)\Big|_{X=\Phi_t^{-1}(x)}, \quad (x,t)\in\overline{Q_T}.
\end{align*}
By the regularity of $\Phi_{(\cdot)}$ and $\Phi_{(\cdot)}^{-1}$, we have
\begin{align} \label{E:Vel_Bd}
  |\mathbf{v}_\Omega| \leq c, \quad |\nabla\mathbf{v}_\Omega| \leq c \quad\text{on}\quad \overline{Q_T}.
\end{align}
For $t\in[0,T]$, let $\bm{\nu}(\cdot,t)$ be the unit outward normal vector field of $\partial\Omega_t$.
We define the scalar outer normal velocity $V_\Omega$ of $\partial\Omega_t$ by $V_\Omega:=\mathbf{v}_\Omega\cdot\bm{\nu}$ on $\partial\Omega_t$.

\subsection{Function spaces} \label{SS:Pre_FS}
Next, we introduce evolving Bochner spaces of functions on $Q_T$.
We follow the abstract framework developed in \cite{AlElSt15_PM,AlCaDjEl23}.
For details of the abstract theory, we refer to these papers.

For functions $U$ on $\Omega_0$ and $u$ on $\Omega_t$, let
\begin{align*}
  \phi_tU := U\circ\Phi_t^{-1} \quad\text{on}\quad \Omega_t, \quad \phi_{-t}u := u\circ\Phi_t \quad\text{on}\quad \Omega_0.
\end{align*}
If $u$ is a function on $Q_T$, then $\phi_{-t}u(t)$ is defined on $\Omega_0$ for $t\in(0,T)$.
When $\phi_{-t}u(t)$ is differentiable in $t$, we define the (strong) material derivative of $u$ on $Q_T$ by
\begin{align*}
  \partial^\bullet u := \phi_t[\partial_t(\phi_{-t}u(t))], \quad\text{i.e.,}\quad \partial^\bullet u(x,t) := \frac{\partial}{\partial t}\Bigl(u(\Phi_t(X),t)\Bigr)\Big|_{X=\Phi_t^{-1}(x)}
\end{align*}
for $(x,t)\in Q_T$.
If $u\in C^1(Q_T)$, then $\partial^\bullet u=\partial_tu+\mathbf{v}_\Omega\cdot\nabla u$ on $Q_T$ by the chain rule.

Let $\mathcal{X}=L^r,W^{1,r}$ with $r\in[1,\infty]$.
By Assumption \ref{A:Domain}, we see that
\begin{align*}
  \phi_t\colon\mathcal{X}(\Omega_0)\to\mathcal{X}(\Omega_t), \quad \phi_{-t}=[\phi_t]^{-1}\colon\mathcal{X}(\Omega_t)\to\mathcal{X}(\Omega_0)
\end{align*}
are bounded linear operators and that the pair $(\mathcal{X}(\Omega_t),\phi_t)_{t\in[0,T]}$ is compatible in the sense of \cite[Assumption 2.1]{AlCaDjEl23}.
We write
\begin{align*}
  \mathcal{X}_T:=\bigcup_{t\in[0,T]}\mathcal{X}(\Omega_t)\times\{t\}, \quad u\colon[0,T]\to\mathcal{X}_T, \quad u(t) = (\hat{u}(t),t)
\end{align*}
and identify $u(t)$ with $\hat{u}(t)$.
For $q\in[1,\infty]$, we define
\begin{align*}
  L_{\mathcal{X}}^q &:= \{u\colon[0,T]\to\mathcal{X}_T \mid \phi_{-(\cdot)}{u}(\cdot)\in L^q(0,T;\mathcal{X}(\Omega_0))\}.
\end{align*}
The space $L_{\mathcal{X}}^q$ is a Banach space equipped with norm
\begin{align} \label{E:Def_LBq}
  \|u\|_{L_{\mathcal{X}}^q} :=
  \begin{cases}
    \left(\int_0^T\|u(t)\|_{\mathcal{X}(\Omega_t)}^q\,dt\right)^{1/q} &\text{if}\quad q\neq\infty, \\
    \esup_{t\in[0,T]}\|u(t)\|_{\mathcal{X}(\Omega_t)} &\text{if}\quad q=\infty.
  \end{cases}
\end{align}
In particular, we may consider $L_{L^q}^q=L^q(Q_T)$.
Moreover,
\begin{align} \label{E:LBq_Equi}
  c^{-1}\|u\|_{L_{\mathcal{X}}^q} \leq \|\phi_{-(\cdot)}u(\cdot)\|_{L^q(0,T;\mathcal{X}(\Omega_0))} \leq c\|u\|_{L_{\mathcal{X}}^q}
\end{align}
for all $u\in L_{\mathcal{X}}^q$ by \eqref{E:Det_Bd} and \eqref{E:Grad_Bd}.
For $k\in\mathbb{Z}_{\geq0}$, we define
\begin{align*}
  C_{\mathcal{X}}^k &:= \{u\colon[0,T]\to\mathcal{X}_T \mid \phi_{-(\cdot)}u(\cdot)\in C^k([0,T];\mathcal{X}(\Omega_0))\} \quad (C^0 = C), \\
  \mathcal{D}_{\mathcal{X}} &:= \{u\colon[0,T]\to\mathcal{X}_T \mid \phi_{-(\cdot)}u(\cdot)\in\mathcal{D}(0,T;\mathcal{X}(\Omega_0))\} \quad (\mathcal{D} = C_c^\infty).
\end{align*}
When $q\neq\infty$, we easily find that $\mathcal{D}_{\mathcal{X}}$ is dense in $L_{\mathcal{X}}^q$ by using \eqref{E:LBq_Equi} and applying cut-off and mollification in time to $\phi_{-(\cdot)}u(\cdot)$.

Let $\phi_t^\ast\colon[\mathcal{X}(\Omega_t)]^\ast\to[\mathcal{X}(\Omega_0)]^\ast$ be the dual operator of $\phi_t\colon\mathcal{X}(\Omega_0)\to\mathcal{X}(\Omega_t)$.
We write
\begin{align*}
  [\mathcal{X}]_T^\ast:=\bigcup_{t\in[0,T]}[\mathcal{X}(\Omega_t)]^\ast\times\{t\}, \quad f\colon[0,T]\to[\mathcal{X}]_T^\ast, \quad f(t) = (\hat{f}(t),t)
\end{align*}
and again identify $f(t)$ with $\hat{f}(t)$.
For $q\in[1,\infty]$, we set
\begin{align*}
  L_{[\mathcal{X}]^\ast}^q := \{f\colon[0,T]\to[\mathcal{X}]_T^\ast \mid \phi_{(\cdot)}^\ast f(\cdot)\in L^q(0,T;[\mathcal{X}(\Omega_0)]^\ast)\}
\end{align*}
and define the norm $\|f\|_{L_{[\mathcal{X}]^\ast}^q}$ as in \eqref{E:Def_LBq}.
If $q\neq\infty$ and $\mathcal{X}(\Omega_t)$ is reflexive, then
\begin{align*}
  [L_{\mathcal{X}}^q]^\ast=L_{[\mathcal{X}]^\ast}^{q'}, \quad \langle f,u\rangle_{L_{\mathcal{X}}^q} = \int_0^T\langle f(t),u(t)\rangle_{\mathcal{X}(\Omega_t)}\,dt, \quad f\in L_{[\mathcal{X}]^\ast}^{q'}, \, u\in L_{\mathcal{X}}^q
\end{align*}
with $1/q+1/q'=1$.
Thus, $L_{\mathcal{X}}^q$ is reflexive if $q\neq1,\infty$.
Also, for all $f\in L_{[\mathcal{X}]^\ast}^{q'}$,
\begin{align} \label{E:Dual_Equi}
  c^{-1}\|f\|_{L_{[\mathcal{X}]^\ast}^{q'}} \leq \|\phi_{(\cdot)}^\ast f(\cdot)\|_{L^{q'}(0,T;[\mathcal{X}(\Omega_0)]^\ast)} \leq c\|f\|_{L_{[\mathcal{X}]^\ast}^{q'}}
\end{align}
by \eqref{E:LBq_Equi}, $L_{[\mathcal{X}]^\ast}^{q'}=[L_{\mathcal{X}}^q]^\ast$, and $L^{q'}(0,T;[\mathcal{X}(\Omega_0)]^\ast)=[L^q(0,T;\mathcal{X}(\Omega_0))]^\ast$.

When $u\in C^1(\overline{Q_T})$, the Reynolds transport theorem (see e.g. \cite{Gur81}) gives
\begin{align*}
  \frac{d}{dt}\int_{\Omega_t}u(t)\,dx = \int_{\Omega_t}\{\partial^\bullet u(t)+[u\,\mathrm{div}\,\mathbf{v}_\Omega](t)\}\,dx, \quad t\in[0,T].
\end{align*}
Based on this formula, we define a weak time derivative of $u\in L_{W^{1,p}}^1$ as follows.

\begin{definition} \label{D:We_MatDe}
  Let $p\in(2,\infty)$.
  We say that $u\in L_{W^{1,p}}^1$ has the weak material derivative if there exists a functional $v\in L_{[W^{1,p}]^\ast}^1$ such that
  \begin{align*}
     \int_0^T\langle v(t),\psi(t)\rangle_{W^{1,p}(\Omega_t)}\,dt = -\int_0^T\bigl(u(t),[\partial^\bullet\psi+\psi\,\mathrm{div}\,\mathbf{v}_\Omega](t)\bigr)_{L^2(\Omega_t)}\,dt
  \end{align*}
  for all $\psi\in\mathcal{D}_{W^{1,p}}$.
  In what follows, we write $v=\partial^\bullet u$.
\end{definition}

\begin{definition} \label{D:TD_SoBo}
  Let $p\in(2,\infty)$ and $p'\in(1,2)$ satisfy $1/p+1/p'=1$.
  We set
  \begin{align*}
    \mathbb{W}^{p,p'} := \{u\in L_{W^{1,p}}^p \mid \partial^\bullet u\in L_{[W^{1,p}]^\ast}^{p'}\}, \quad \|u\|_{\mathbb{W}^{p,p'}} := \|u\|_{L_{W^{1,p}}^p}+\|\partial^\bullet u\|_{L_{[W^{1,p}]^\ast}^{p'}}.
  \end{align*}
  In \cite[Definition 3.17]{AlCaDjEl23}, it is written as $\mathbb{W}^{p,p'}=\mathbb{W}^{p,p'}(W^{1,p},[W^{1,p}]^\ast)$.
\end{definition}

The space $\mathbb{W}^{p,p'}$ is a Banach space.
Since we have the Gelfand triple structure
\begin{align*}
  W^{1,p}(\Omega_t) \hookrightarrow L^2(\Omega_t) \hookrightarrow [W^{1,p}(\Omega_t)]^\ast,
\end{align*}
we can observe as in \cite[Propisition 6.5]{AlCaDjEl23} that, under Assumption \ref{A:Domain}, there is an evolving space equivalence between the spaces
\begin{align} \label{E:EvSE}
  \mathbb{W}^{p,p'} \quad\text{and}\quad \{U\in L^p(0,T;W^{1,p}(\Omega_0)) \mid \partial_tU\in L^{p'}(0,T;[W^{1,p}(\Omega_0)]^\ast)\}
\end{align}
in the sense of \cite[Definition 3.19]{AlCaDjEl23}.
Hence, the following results are available.

\begin{lemma} \label{L:Trans}
  The embedding $\mathbb{W}^{p,p'}\hookrightarrow C_{L^2}$ is continuous, where
  \begin{align*}
    \|u\|_{C_{L^2}} := \sup_{t\in[0,T]}\|u(t)\|_{L^2(\Omega_t)}, \quad u\in C_{L^2}.
  \end{align*}
  Moreover, for all $u_1,u_2\in\mathbb{W}^{p,p'}$ and almost all $t\in(0,T)$, we have
  \begin{multline} \label{E:Trans}
    \frac{d}{dt}\bigl(u_1(t),u_2(t)\bigr)_{L^2(\Omega_t)} = \langle \partial^\bullet u_1(t),u_2(t)\rangle_{W^{1,p}(\Omega_t)}+\langle \partial^\bullet u_2(t),u_1(t)\rangle_{W^{1,p}(\Omega_t)} \\
    +\bigl(u_1(t),[u_2\,\mathrm{div}\,\mathbf{v}_\Omega](t)\bigr)_{L^2(\Omega_t)}.
  \end{multline}
\end{lemma}

\begin{proof}
  The statements follow from the abstract results given in \cite[Theorem 4.5]{AlCaDjEl23}.
\end{proof}

Let us also give two auxiliary results on the weak material derivative.

\begin{lemma} \label{L:Wpp_Mult}
  Let $\chi\in C^1(\overline{Q_T})$ and $u\in\mathbb{W}^{p,p'}$.
  Then, $\chi u\in\mathbb{W}^{p,p'}$ and
  \begin{align} \label{E:WeMa_Mult}
    \langle[\partial^\bullet(\chi u)](t),\psi(t)\rangle_{W^{1,p}(\Omega_t)} = \langle\partial^\bullet u(t),[\chi\psi](t)\rangle_{W^{1,p}(\Omega_t)}+\bigl(u(t),[\psi\partial^\bullet\chi](t)\bigr)_{L^2(\Omega_t)}
  \end{align}
  for all $\psi\in L_{W^{1,p}}^p$ and almost all $t\in(0,T)$.
  Moreover, for a.a. $t\in(0,T)$,
  \begin{multline} \label{E:Tr_Mult}
    \frac{d}{dt}\bigl(u(t),[\chi u](t)\bigr)_{L^2(\Omega_t)} = 2\langle\partial^\bullet u(t),[\chi u](t)\rangle_{W^{1,p}(\Omega_t)}+\bigl(u(t),[u\partial^\bullet\chi](t)\bigr)_{L^2(\Omega_t)} \\
    +\bigl(u(t),[\chi u\,\mathrm{div}\,\mathbf{v}_\Omega](t)\bigr)_{L^2(\Omega_t)}.
  \end{multline}
\end{lemma}

\begin{proof}
  Since $\chi\in C^1(\overline{Q_T})$, we have $\chi u\in L_{W^{1,p}}^p$.
  Next, let
  \begin{align*}
    L(\psi) := -\int_0^T\bigl([\chi u](t),[\partial^\bullet\psi+\psi\,\mathrm{div}\,\mathbf{v}_\Omega](t)\bigr)_{L^2(\Omega_t)}\,dt, \quad \psi\in\mathcal{D}_{W^{1,p}}.
  \end{align*}
  Then, $L(\psi)$ is linear in $\psi$.
  Moreover, since $\chi\partial^\bullet\psi=\partial^\bullet(\chi\psi)-\psi\partial^\bullet\chi$ and $u\in\mathbb{W}^{p,p'}$,
  \begin{align*}
      L(\psi) &= -\int_0^T\bigl(u(t),[\partial^\bullet(\chi\psi)+\chi\psi\,\mathrm{div}\,\mathbf{v}_\Omega](t)\bigr)_{L^2(\Omega_t)}\,dt+\int_0^T\bigl(u(t),[\psi\partial^\bullet\chi](t)\bigr)_{L^2(\Omega_t)}\,dt \\
      &= \int_0^T\langle\partial^\bullet u(t),[\chi\psi](t)\rangle_{W^{1,p}(\Omega_t)}\,dt+\int_0^T\bigl(u(t),[\psi\partial^\bullet\chi](t)\bigr)_{L^2(\Omega_t)}\,dt.
  \end{align*}
  Hence, by $\chi\in C^1(\overline{Q_T})$, \eqref{E:Vel_Bd}, H\"{o}lder's inequality, $p>2$, \eqref{E:Volume}, and $T<\infty$, we have
  \begin{align*}
    |L(\psi)| \leq \|\partial^\bullet u\|_{L_{[W^{1,p}]^\ast}^{p'}}\|\chi\psi\|_{L_{W^{1,p}}^p}+\|u\|_{L_{L^2}^2}\|\psi\partial^\bullet\chi\|_{L_{L^2}^2} \leq c_T\|u\|_{\mathbb{W}^{p,p'}}\|\psi\|_{L_{W^{1,p}}^p}
  \end{align*}
  with some constant $c_T>0$ depending only on $T$.
  Since $\mathcal{D}_{W^{1,p}}$ is dense in $L_{W^{1,p}}^p$, it follows that $L$ extends to a bounded linear functional on $L_{W^{1,p}}^p$.
  Thus,
  \begin{align*}
    \partial^\bullet(\chi u) \in [L_{W^{1,p}}^p]^\ast = L_{[W^{1,p}]^\ast}^{p'}, \quad \chi u\in\mathbb{W}^{p,p'}.
  \end{align*}
  Moreover, by the above relation for $L(\psi)$ and a density argument, we have
  \begin{align*}
    &\int_0^T\langle[\partial^\bullet(\chi u)](t),\psi(t)\rangle_{W^{1,p}(\Omega_t)}\,dt \\
    &\qquad = L(\psi) = \int_0^T\Bigl(\langle\partial^\bullet u(t),[\chi\psi](t)\rangle_{W^{1,p}(\Omega_t)}+\bigl(u(t),[\psi\partial^\bullet\chi](t)\bigr)_{L^2(\Omega_t)}\Bigr)\,dt
  \end{align*}
  for all $\psi\in L_{W^{1,p}}^p$.
  Thus, replacing $\psi$ by $\theta\psi$ with any $\theta\in C_c^1(0,T)$, we get \eqref{E:WeMa_Mult}.
  Also, we have \eqref{E:Tr_Mult} by \eqref{E:Trans} and \eqref{E:WeMa_Mult}.
\end{proof}

\begin{lemma} \label{L:Wpp_Comp}
  Let $\zeta\in C^2(\mathbb{R})$ such that $\zeta''$ is bounded on $\mathbb{R}$.
  Then,
  \begin{align*}
    \zeta'(u) = \zeta'\circ u \in L_{W^{1,p}}^p, \quad \zeta(u) = \zeta\circ u \in C_{L^1} \quad\text{for all}\quad u\in \mathbb{W}^{p,p'}.
  \end{align*}
  Moreover, for almost all $t\in(0,T)$, we have
  \begin{align} \label{E:Wpp_Comp}
    \frac{d}{dt}\int_{\Omega_t}[\zeta(u)](t)\,dx = \langle\partial^\bullet u(t),[\zeta'(u)](t)\rangle_{W^{1,p}(\Omega_t)}+\int_{\Omega_t}[\zeta(u)\,\mathrm{div}\,\mathbf{v}_\Omega](t)\,dx.
  \end{align}
\end{lemma}

\begin{proof}
  Let $u\in\mathbb{W}^{p,p'}\subset C_{L^2}$ (see Lemma \ref{L:Trans}).
  Since $\zeta''$ is bounded on $\mathbb{R}$,
  \begin{align} \label{Pf_WC:dsze}
    |\zeta'(s)| = \left|\zeta'(0)+\int_0^s\zeta''(\tau)\,d\tau\right| \leq c(1+|s|), \quad s\in\mathbb{R}.
  \end{align}
  By this fact, \eqref{E:Volume}, $T<\infty$, and $\nabla[\zeta'(u)]=\zeta''(u)\nabla u$, we get $\zeta'(u)\in L_{W^{1,p}}^p$.
  Also, let
  \begin{align*}
    \left\{
    \begin{aligned}
      U(X,t) &:= [\phi_{-t}u(t)](X) = u(\Phi_t(X),t), \\
      Z(X,t) &:= \bigl[\phi_{-t}[\zeta(u)](t)\bigr](X) = \zeta\bigl(u(\Phi_t(X),t)\bigr),
    \end{aligned}
    \right. \quad (X,t) \in \Omega_0\times[0,T].
  \end{align*}
  Then, $U\in C([0,T];L^2(\Omega_0))$ by $u\in C_{L^2}$.
  Since $Z=\zeta(U)$, we have
  \begin{align*}
    |Z(t)-Z(t_0)| \leq c(1+|U(t)|+|U(t_0)|)|U(t)-U(t_0)| \quad\text{in}\quad \Omega_0
  \end{align*}
  for all $t,t_0\in[0,T]$ by the mean value theorem for $\zeta$ and \eqref{Pf_WC:dsze}.
  Thus, setting
  \begin{align*}
    c_U := c\Bigl(|\Omega_0|^{1/2}+2\|U\|_{C([0,T];L^2(\Omega_0))}\Bigr) > 0,
  \end{align*}
  we observe by H\"{o}lder's inequality and $U\in C([0,T];L^2(\Omega_0))$ that
  \begin{align*}
    \|Z(t)-Z(t_0)\|_{L^1(\Omega_0)} \leq c_U\|U(t)-U(t_0)\|_{L^2(\Omega_0)} \to 0 \quad\text{as}\quad t\to t_0.
  \end{align*}
  Hence, $Z=\phi_{-(\cdot)}[\zeta(u)](\cdot)\in C([0,T];L^1(\Omega_0))$ and $\zeta(u)\in C_{L^1}$.

  Next, we prove the following claim: if $u_k\to u$ strongly in $L_{W^{1,p}}^p$, then there exists a subsequence of $\{u_k\}$, which is denoted by $\{u_k\}$ again, such that
  \begin{align} \label{Pf_WC:zeta}
    \lim_{k\to\infty}\|\zeta(u)-\zeta(u_k)\|_{L_{L^1}^1} = 0, \quad \lim_{k\to\infty}\|\zeta'(u)-\zeta'(u_k)\|_{L_{W^{1,p}}^p} = 0.
  \end{align}
  By the mean value theorem for $\zeta$ and $\zeta'$, \eqref{Pf_WC:dsze}, and $|\zeta''|\leq c$ on $\mathbb{R}$, we have
  \begin{align*}
    |\zeta(u)-\zeta(u_k)| \leq c(1+|u|+|u_k|)|u-u_k|, \quad |\zeta'(u)-\zeta'(u_k)| \leq c|u-u_k|
  \end{align*}
  in $Q_T$.
  Thus, we see by H\"{o}lder's inequality, $p>2$, and \eqref{E:Volume} that
  \begin{align*}
    \|\zeta(u)-\zeta(u_k)\|_{L_{L^1}^1} &\leq c\Bigl(\|1\|_{L_{L^2}^2}+\|u\|_{L_{L^2}^2}+\|u_k\|_{L_{L^2}^2}\Bigr)\|u-u_k\|_{L_{L^2}^2} \\
    &\leq c_T\Bigl(1+\|u\|_{L_{L^p}^p}+\|u_k\|_{L_{L^p}^p}\Bigr)\|u-u_k\|_{L_{L^p}^p},
  \end{align*}
  where $c_T>0$ is a constant depending only on $T$, and that
  \begin{align*}
    \|\zeta'(u)-\zeta'(u_k)\|_{L_{L^p}^p} \leq c\|u-u_k\|_{L_{L^p}^p}.
  \end{align*}
  Since $u_k\to u$ strongly in $L_{W^{1,p}}^p$, it follows that
  \begin{align*}
    \lim_{k\to\infty}\|\zeta(u)-\zeta(u_k)\|_{L_{L^1}^1} = 0, \quad \lim_{k\to\infty}\|\zeta'(u)-\zeta'(u_k)\|_{L_{L^p}^p} = 0.
  \end{align*}
  Also, since $u_k\to u$ and $\nabla u_k\to\nabla u$ strongly in $L_{L^p}^p=L^p(Q_T)$, we can take a subsequence of $\{u_k\}$, which is denoted by $\{u_k\}$ again, and a function $\varphi\in L^p(Q_T)$ such that
  \begin{align*}
    |\nabla u_k| \leq \varphi, \quad \lim_{k\to\infty}u_k = u, \quad \lim_{k\to\infty}\nabla u_k = \nabla u \quad\text{a.e. in}\quad Q_T
  \end{align*}
  as in the proof of the completeness of $L^p(Q_T)$.
  Then,
  \begin{align*}
    \lim_{k\to\infty}\nabla[\zeta'(u_k)] = \lim_{k\to\infty}\zeta''(u_k)\nabla u_k = \zeta''(u)\nabla u= \nabla[\zeta'(u)] \quad\text{a.e. in}\quad Q_T
  \end{align*}
  by $\zeta\in C^2(\mathbb{R})$.
  Moreover, since $|\zeta''|\leq c$ on $\mathbb{R}$, we have
  \begin{align*}
    |\nabla[\zeta'(u)]-\nabla[\zeta'(u_k)]|^p \leq c(|\nabla u|^p+|\nabla u_k|^p) \leq c(|\nabla u|^p+\varphi^p) \quad\text{a.e. in}\quad Q_T,
  \end{align*}
  where the right-hand side is in $L^1(Q_T)$ and independent of $k$.
  Thus,
  \begin{align*}
    \lim_{k\to\infty}\|\nabla[\zeta'(u)]-\nabla[\zeta'(u_k)]\|_{L_{L^p}^p} = \lim_{k\to\infty}\|\nabla[\zeta'(u)]-\nabla[\zeta'(u_k)]\|_{L^p(Q_T)} = 0
  \end{align*}
  by the dominated convergence theorem, and \eqref{Pf_WC:zeta} is valid.

  Now, Let us show \eqref{E:Wpp_Comp} for $u\in\mathbb{W}^{p,p'}$.
  To this end, it suffices to verify
  \begin{align} \label{Pf_WC:Goal}
    \begin{aligned}
      &\int_0^T\theta'(t)\left(\int_{\Omega_t}[\zeta(u)](t)\,dx\right)\,dt \\
      &\qquad = \int_0^T\theta(t)\left\{\langle\partial^\bullet u(t),[\zeta'(u)](t)\rangle_{W^{1,p}(\Omega_t)}+\int_{\Omega_t}[\zeta(u)\,\mathrm{div}\,\mathbf{v}_\Omega](t)\,dx\right\}\,dt
    \end{aligned}
  \end{align}
  for all $\theta\in C_c^1(0,T)$.
  We prove this in three steps.

  \textbf{Step 1:} let $u\in C^1(\overline{Q_T})$.
  Then, since $\zeta(u)\in C^1(\overline{Q_T})$, we have
  \begin{align*}
    \frac{d}{dt}\int_{\Omega_t}[\zeta(u)](t)\,dx = \int_{\Omega_t}\{[\zeta'(u)\partial^\bullet u](t)+[\zeta(u)\,\mathrm{div}\,\mathbf{v}_\Omega](t)\}\,dx, \quad t\in[0,T]
  \end{align*}
  by the Reynorlds transport theorem and the chain rule.
  Thus, by multiplying both sides by $\theta$ and carrying out integration by parts in time, we find that \eqref{Pf_WC:Goal} holds with $\langle\cdot,\cdot\rangle_{W^{1,p}(\Omega_t)}$ replaced by $(\cdot,\cdot)_{L^2(\Omega_t)}$.

  \textbf{Step 2:} let $u\in C_{W^{1,p}}^1$.
  Since
  \begin{align*}
    U := \phi_{-(\cdot)}u(\cdot) \in C^1([0,T];W^{1,p}(\Omega_0)) \subset W^{1,p}(\Omega_0\times(0,T))
  \end{align*}
  and $\Omega_0\times(0,T)$ is a bounded Lipschitz domain, we can take functions
  \begin{align*}
    U_k \in C^1(\overline{\Omega_0}\times[0,T]) \quad\text{such that}\quad \lim_{k\to\infty}\|U-U_k\|_{W^{1,p}(\Omega_0\times(0,T))} = 0.
  \end{align*}
  Note that the strong convergence in $W^{1,p}(\Omega_0\times(0,T))$ means that
  \begin{align} \label{Pf_WC:S2zero}
    \lim_{k\to\infty}\|U-U_k\|_{L^p(0,T;W^{1,p}(\Omega_0))} = 0, \quad \lim_{k\to\infty}\|\partial_tU-\partial_tU_k\|_{L^p(0,T;L^p(\Omega_0))} = 0.
  \end{align}
  Let $u_k:=\phi_{(\cdot)}U_k(\cdot)\in C^1(\overline{Q_T})$.
  Then, by \eqref{E:LBq_Equi}, $\partial_tU=\phi_{-(\cdot)}[\partial^\bullet u](\cdot)$, and \eqref{Pf_WC:S2zero},
  \begin{align} \label{Pf_WC:S2Ap}
    \lim_{k\to\infty}\|u-u_k\|_{L_{W^{1,p}}^p} = 0, \quad \lim_{k\to\infty}\|\partial^\bullet u-\partial^\bullet u_k\|_{L_{L^p}^p} = 0.
  \end{align}
  Moreover, \eqref{Pf_WC:zeta} holds up to a subsequence.
  By Step 1, $u_k$ satisfies \eqref{Pf_WC:Goal} with $\langle\cdot,\cdot\rangle_{W^{1,p}(\Omega_t)}$ replaced by $(\cdot,\cdot)_{L^2(\Omega_t)}$.
  Thus, we see that the same equality holds for $u$ by letting $k\to\infty$ and using H\"{o}lder's inequality, $p>2$, \eqref{E:Vel_Bd}, \eqref{Pf_WC:zeta}, and \eqref{Pf_WC:S2Ap}.

  \textbf{Step 3:} let $u\in\mathbb{W}^{p,p'}$.
  Since there is an evolving space equivalence between the spaces \eqref{E:EvSE}, we can use \cite[Lemma 3.20]{AlCaDjEl23} to observe that $C_{W^{1,p}}^1$ is dense in $\mathbb{W}^{p,p'}$.
  Thus, we can take functions $w_k\in C_{W^{1,p}}^1$ such that
  \begin{align} \label{Pf_WC:S3}
    \lim_{k\to\infty}\|u-w_k\|_{L_{W^{1,p}}^p} = 0, \quad \lim_{k\to\infty}\|\partial^\bullet u-\partial^\bullet w_k\|_{L_{[W^{1,p}]^\ast}^{p'}} = 0,
  \end{align}
  and \eqref{Pf_WC:zeta} holds with $u_k$ replaced by $w_k$ up to a subsequence.
  Since $w_k$ satisfies \eqref{Pf_WC:Goal} by Step 2, we send $k\to\infty$ and use H\"{o}lder's inequality, $p>2$, \eqref{E:Vel_Bd}, \eqref{Pf_WC:zeta}, and \eqref{Pf_WC:S3} to find that $u$ also satisfies \eqref{Pf_WC:Goal}.
  This completes the proof.
\end{proof}

\subsection{Other basic results} \label{SS:Pre_BaRe}
Let $p\in(2,\infty)$.
We give some other basic results.

\begin{lemma} \label{L:Ine_Vec}
  Let $K\in\mathbb{N}$.
  For all $\mathbf{a},\mathbf{b}\in\mathbb{R}^K$, we have
  \begin{align}
    &(|\mathbf{a}|^{p-2}\mathbf{a}-|\mathbf{b}|^{p-2}\mathbf{b})\cdot(\mathbf{a}-\mathbf{b}) \geq 0, \label{E:Vec_Mono} \\
    &\bigl|\,|\mathbf{a}|^{p-2}\mathbf{a}-|\mathbf{b}|^{p-2}\mathbf{b}\,\bigr| \leq c(|\mathbf{a}|^{p-2}+|\mathbf{b}|^{p-2})|\mathbf{a}-\mathbf{b}|. \label{E:p_Lip}
  \end{align}
\end{lemma}

\begin{proof}
  We see by direct calculations and $|\mathbf{a}\cdot\mathbf{b}|\leq|\mathbf{a}|\cdot|\mathbf{b}|$ that
  \begin{align*}
    (|\mathbf{a}|^{p-2}\mathbf{a}-|\mathbf{b}|^{p-2}\mathbf{b})\cdot(\mathbf{a}-\mathbf{b}) \geq (|\mathbf{a}|^{p-1}-|\mathbf{b}|^{p-1})(|\mathbf{a}|-|\mathbf{b}|),
  \end{align*}
  where the right-hand side is nonnegative by $p-1>0$.
  Thus, \eqref{E:Vec_Mono} is valid.

  Next, we prove \eqref{E:p_Lip}.
  We first observe that
  \begin{align} \label{Pf_IV:ab}
    (a+b)^\alpha \leq \max\{(2a)^\alpha,(2b)^\alpha\} \leq 2^\alpha(a^\alpha+b^\alpha), \quad a,b,\alpha>0.
  \end{align}
  Let $\mathbf{x}(s):=s\mathbf{a}+(1-s)\mathbf{b}$ for $s\in\mathbb{R}$.
  Since $\mathbf{x}(1)=\mathbf{a}$ and $\mathbf{x}(0)=\mathbf{b}$,
  \begin{align} \label{Pf_IV:Int}
    |\mathbf{a}|^{p-2}\mathbf{a}-|\mathbf{b}|^{p-2}\mathbf{b} = \int_0^1\frac{d}{ds}\Bigl(|\mathbf{x}(s)|^{p-2}\mathbf{x}(s)\Bigr)\,ds.
  \end{align}
  When $\mathbf{x}(s)\neq\mathbf{0}$,
  \begin{align*}
    \frac{d}{ds}\Bigl(|\mathbf{x}(s)|^{p-2}\mathbf{x}(s)\Bigr) = (p-2)|\mathbf{x}(s)|^{p-4}\{\mathbf{x}(s)\cdot\mathbf{x}'(s)\}\mathbf{x}(s)+|\mathbf{x}(s)|^{p-2}\mathbf{x}'(s).
  \end{align*}
  Thus, by $|\mathbf{x}(s)|\leq|\mathbf{a}|+|\mathbf{b}|$, \eqref{Pf_IV:ab}, and $\mathbf{x}'(s)=\mathbf{a}-\mathbf{b}$,
  \begin{align} \label{Pf_IV:Der}
    \begin{aligned}
      \left|\frac{d}{ds}\Bigl(|\mathbf{x}(s)|^{p-2}\mathbf{x}(s)\Bigr)\right| &\leq (p-1)|\mathbf{x}(s)|^{p-2}|\mathbf{x}'(s)| \\
      &\leq 2^{p-2}(p-1)(|\mathbf{a}|^{p-2}+|\mathbf{b}|^{p-2})|\mathbf{a}-\mathbf{b}|.
    \end{aligned}
  \end{align}
  Also, if $\mathbf{x}(s)=s\mathbf{a}+(1-s)\mathbf{b}=\mathbf{0}$, then $\mathbf{x}(s+h)=h(\mathbf{a}-\mathbf{b})$ and
  \begin{align*}
    \frac{d}{ds}\Bigl(|\mathbf{x}(s)|^{p-2}\mathbf{x}(s)\Bigr) &= \lim_{h\to0}\frac{|\mathbf{x}(s+h)|^{p-2}\mathbf{x}(s+h)-|\mathbf{x}(s)|^{p-2}\mathbf{x}(s)}{h} \\
    &= \lim_{h\to0}|h(\mathbf{a}-\mathbf{b})|^{p-2}(\mathbf{a}-\mathbf{b}) = \mathbf{0},
  \end{align*}
  and \eqref{Pf_IV:Der} also holds in this case.
  Thus, applying \eqref{Pf_IV:Der} to \eqref{Pf_IV:Int}, we get \eqref{E:p_Lip}.
\end{proof}

\begin{lemma} \label{L:Lp_L2}
  Let $t\in[0,T]$ and $u\in L^p(\Omega_t)$.
  Then, for any $\delta>0$,
  \begin{align} \label{E:Lp_L2}
    \|u\|_{L^2(\Omega_t)}^2 \leq \delta\|u\|_{L^p(\Omega_t)}^p+c_\delta,
  \end{align}
  where $c_\delta>0$ is a constant depending only on $\delta$ (and $p$).
\end{lemma}

\begin{proof}
  Let $\sigma>0$.
  Since $p>2$, we can take a $q>1$ such that $2/p+1/q=1$.
  Then,
  \begin{align*}
    \int_{\Omega_t}|u|^2\,dx &\leq \left(\int_{\Omega_t}|u|^p\,dx\right)^{2/p}|\Omega_t|^{1/q} = \left(\sigma\int_{\Omega_t}|u|^p\,dx\right)^{2/p}\cdot\sigma^{-2/p}|\Omega_t|^{1/q} \\
    &\leq \frac{2\sigma}{p}\int_{\Omega_t}|u|^p\,dx+\frac{\sigma^{-2q/p}|\Omega_t|}{q}
  \end{align*}
  by H\"{o}lder's and Young's inequality.
  Replacing $2\sigma/p$ by $\delta$ and using \eqref{E:Volume}, we get \eqref{E:Lp_L2}.
\end{proof}

Next, we prove the uniform-in-time Poincar\'{e} inequality on $\Omega_t$.

\begin{lemma} \label{L:Un_Poin}
  Let $q\in[1,\infty]$.
  There exists a constant $c>0$ such that
  \begin{align} \label{E:Un_Poin}
    \|u\|_{L^q(\Omega_t)} \leq c\|\nabla u\|_{L^q(\Omega_t)}
  \end{align}
  for all $t\in[0,T]$ and $u\in W^{1,q}(\Omega_t)$ satisfying $\int_{\Omega_t}u\,dx=0$.
\end{lemma}

\begin{proof}
  We show \eqref{E:Un_Poin} by contradiction.
  Assume to the contrary that for each $k\in\mathbb{N}$ there exist some $t_k\in[0,T]$ and $u_k\in W^{1,q}(\Omega_{t_k})$ such that
  \begin{align*}
    \int_{\Omega_{t_k}}u_k(x)\,dx = 0, \quad \|u_k\|_{L^q(\Omega_{t_k})} > k\|\nabla u_k\|_{L^q(\Omega_{t_k})}.
  \end{align*}
  Let $U_k:=\phi_{-t_k}u_k=u_k\circ\Phi_{t_k}$ on $\Omega_0$.
  Then, we see by \eqref{E:Det_Bd} and \eqref{E:Grad_Bd} that
  \begin{align*}
    \int_{\Omega_0}U_k(X)J_{t_k}(X)\,dX = 0, \quad \|U_k\|_{L^q(\Omega_0)} > ck\|\nabla U_k\|_{L^q(\Omega_0)}.
  \end{align*}
  Since $\|U_k\|_{L^q(\Omega_0)}\neq0$, we can replace $U_k$ by $U_k/\|U_k\|_{L^q(\Omega_0)}$ to assume that
  \begin{align} \label{Pf_UP:Uk}
    \int_{\Omega_0}U_k(X)J_{t_k}(X)\,dX = 0, \quad \|U_k\|_{L^q(\Omega_0)} = 1, \quad \|\nabla U_k\|_{L^q(\Omega_0)} < \frac{1}{ck}.
  \end{align}
  Thus, $\{U_k\}_k$ is bounded in $W^{1,q}(\Omega_0)$.
  Since the embedding $W^{1,q}(\Omega_0)\hookrightarrow L^q(\Omega_0)$ is compact, and since $\{t_k\}_k$ is a sequence in the finite interval $[0,T]$, it follows that
  \begin{align} \label{Pf_UP:Conv}
    \lim_{k\to\infty}t_k = t_\infty, \quad \lim_{k\to\infty}U_k = U_\infty \quad\text{strongly in $L^q(\Omega_0)$}
  \end{align}
  up to subsequences with some $t_\infty\in[0,T]$ and $U_\infty\in L^q(\Omega_0)$.
  Then, for all $\psi\in\mathcal{D}(\Omega_0)$ and $i=1,\dots,n$, we observe by H\"{o}lder's inequality, \eqref{Pf_UP:Uk}, and \eqref{Pf_UP:Conv} that
  \begin{align*}
    (U_\infty,\partial_i\psi)_{L^2(\Omega_0)} = \lim_{k\to\infty}(U_k,\partial_i\psi)_{L^2(\Omega_0)} = -\lim_{k\to\infty}(\partial_iU_k,\psi)_{L^2(\Omega_0)} = 0.
  \end{align*}
  Thus, $\nabla U_\infty=0$ on $\Omega_0$ and $U_\infty$ is constant on $\Omega_0$.
  Moreover, noting that $J_{(\cdot)}$ is continuous on $\overline{\Omega_0}\times[0,T]$ by Assumption \ref{A:Domain}, we deduce from \eqref{Pf_UP:Uk} and \eqref{Pf_UP:Conv} that
  \begin{align*}
    U_\infty\int_{\Omega_0}J_{t_\infty}(X)\,dX = \int_{\Omega_0}U_\infty J_{t_\infty}(X)\,dX = \lim_{k\to\infty}\int_{\Omega_0}U_k(X)J_{t_k}(X)\,dX = 0.
  \end{align*}
  By this result and \eqref{E:Det_Bd}, we find that $U_\infty=0$, which contradicts
  \begin{align*}
    \|U_\infty\|_{L^q(\Omega_0)} = \lim_{k\to\infty}\|U_k\|_{L^q(\Omega_0)} = 1.
  \end{align*}
  Therefore, \eqref{E:Un_Poin} is valid.
\end{proof}

\begin{lemma} \label{L:Cor_UP}
  Let $q\in[1,\infty]$.
  There exists a constant $c>0$ such that
  \begin{align} \label{E:Cor_UP}
    \|u\|_{L^q(\Omega_t)} \leq c\Bigl(\|\nabla u\|_{L^q(\Omega_t)}+|(u,1)_{L^2(\Omega_t)}|\Bigr)
  \end{align}
  for all $t\in[0,T]$ and $u\in W^{1,q}(\Omega_t)$.
\end{lemma}

\begin{proof}
  Let $c_u:=|\Omega_t|^{-1}(u,1)_{L^2(\Omega_t)}$ and $v:=u-c_u$ on $\Omega_t$.
  Then,
  \begin{align*}
    \|v\|_{L^q(\Omega_t)} \leq c\|\nabla v\|_{L^q(\Omega_t)} = c\|\nabla u\|_{L^q(\Omega_t)}
  \end{align*}
  by \eqref{E:Un_Poin}, since $v\in W^{1,q}(\Omega_t)$ and $\int_{\Omega_t}v\,dx=0$.
  Hence,
  \begin{align*}
    \|u\|_{L^q(\Omega_t)} &\leq \|v\|_{L^q(\Omega_t)}+\|c_u\|_{L^q(\Omega_t)} = \|v\|_{L^q(\Omega_t)}+|\Omega_t|^{1/q}|c_u| \\
    &\leq c\|\nabla u\|_{L^q(\Omega_t)}+|\Omega_t|^{-1+1/q}|(u,1)_{L^2(\Omega_t)}|,
  \end{align*}
  and we apply \eqref{E:Volume} to the last term to obtain \eqref{E:Cor_UP}.
\end{proof}

It is well known (see e.g. \cite[Lemma 0.5]{Tes14}) that any subset of a separable metric space is again separable.
Also, the following statement holds, which is useful when we take basis functions in the Galerkin method (see Lemma \ref{L:Basis}).

\begin{lemma} \label{L:Separa}
  Let $(\mathcal{M},d_{\mathcal{M}})$ be a separable metric space and $\mathcal{A}$ be a dense subset of $\mathcal{M}$.
  Then, there exists a countable subset $\mathcal{A}_0$ of $\mathcal{A}$ that is also dense in $\mathcal{M}$.
\end{lemma}

This seems to be also well known and the proof is very similar to that of the separability of a subset of a separable space.
We give the proof below for the reader's convenience.

\begin{proof}
  Let $\mathcal{M}_0=\{x_1,x_2,\dots\}$ be a countable dense subset of $\mathcal{M}$, and let
  \begin{align*}
    B_r(x) &:= \{y\in\mathcal{M} \mid d_{\mathcal{M}}(y,x)<r\}, \quad x\in\mathcal{M}, \quad r>0, \\
    \mathbb{S} &:= \{(n,k)\in\mathbb{N}^2 \mid B_{1/k}(x_n)\cap\mathcal{A} \neq \emptyset\}.
  \end{align*}
  If $(n,k)\in\mathbb{S}$, we take any $y_{n,k}\in B_{1/k}(x_n)\cap\mathcal{A}$ and set $\mathcal{A}_0:=\{y_{n,k}\mid(n,k)\in\mathbb{S}\}$.
  Clearly, $\mathcal{A}_0$ is a countable subset of $\mathcal{A}$.
  Let us show that $\mathcal{A}_0$ is dense in $\mathcal{M}$.
  Take any $x\in\mathcal{M}$ and $\varepsilon>0$.
  We fix a sufficiently large $k\in\mathbb{N}$ such that $1/k<\varepsilon/3$.
  Since $\mathcal{A}$ is dense in $\mathcal{M}$, there exists some $a\in\mathcal{A}$ such that $d_{\mathcal{M}}(x,a)<1/k$.
  Next, since $\mathcal{M}_0=\{x_1,x_2,\dots\}$ is dense in $\mathcal{M}$, we can take some $n\in\mathbb{N}$ such that
  \begin{align*}
    d_{\mathcal{M}}(a,x_n) < \frac{1}{k}, \quad\text{i.e.,}\quad a \in B_{1/k}(x_n)\cap\mathcal{A}, \quad B_{1/k}(x_n)\cap\mathcal{A} \neq\emptyset.
  \end{align*}
  Hence, $(n,k)\in\mathbb{S}$ and $y_{n,k}\in\mathcal{A}_0$ exists (note that $a\neq y_{n,k}$ in general).
  We find that
  \begin{align*}
    d_{\mathcal{M}}(x,y_{n,k}) \leq d_{\mathcal{M}}(x,a)+d_{\mathcal{M}}(a,x_n)+d_{\mathcal{M}}(x_n,y_{n,k}) < \frac{3}{k} < \varepsilon
  \end{align*}
  by the above inequalities and $y_{n,k}\in B_{1/k}(x_n)\cap\mathcal{A}$.
  Hence, $\mathcal{A}_0$ is dense in $\mathcal{M}$.
\end{proof}

We will also use the next density result on functions with vanishing mean.

\begin{lemma} \label{L:VM_Den}
  Let $\Omega$ be a bounded Lipschitz domain in $\mathbb{R}^n$ and $q\in[1,\infty)$, and let
  \begin{align} \label{E:Def_VM}
    \begin{aligned}
      W_{\mathrm{vm}}^{1,q}(\Omega) &:= \{u\in W^{1,q}(\Omega) \mid (u,1)_{L^2(\Omega)} = 0\}, \\
      C_{\mathrm{vm}}^\infty(\overline{\Omega}) &:= \{u\in C^\infty(\overline{\Omega}) \mid (u,1)_{L^2(\Omega)} = 0\}.
    \end{aligned}
  \end{align}
  Then, $C_{\mathrm{vm}}^\infty(\overline{\Omega})$ is dense in $W_{\mathrm{vm}}^{1,q}(\Omega)$ with respect to the $W^{1,q}(\Omega)$-norm.
\end{lemma}

\begin{proof}
  Let $u\in W_{\mathrm{vm}}^{1,q}(\Omega)$.
  Since $u\in W^{1,q}(\Omega)$, we can take functions
  \begin{align*}
    v_k \in C^\infty(\overline{\Omega}) \quad\text{such that}\quad \lim_{k\to\infty}\|u-v_k\|_{W^{1,q}(\Omega)} = 0,
  \end{align*}
  see e.g. \cite[Theorem 3.22]{AdaFou03}.
  Let $u_k:=v_k-|\Omega|^{-1}(v_k,1)_{L^2(\Omega)}$ on $\overline{\Omega}$, where $|\Omega|$ is the volume of $\Omega$.
  Then, $u_k\in C_{\mathrm{vm}}^\infty(\overline{\Omega})$.
  Moreover, since $(u,1)_{L^2(\Omega)}=0$ by $u\in W_{\mathrm{vm}}^{1,q}(\Omega)$,
  \begin{align*}
    u-u_k = u-v_k-|\Omega|^{-1}(u-v_k,1)_{L^2(\Omega)} \quad\text{on}\quad \Omega.
  \end{align*}
  Noting that $|\Omega|^{-1}(u-v_k,1)_{L^2(\Omega)}$ is a constant, we see that
  \begin{align*}
    \|u-u_k\|_{L^q(\Omega)} \leq \|u-v_k\|_{L^q(\Omega)}+|\Omega|^{-1+1/q}|(u-v_k,1)_{L^2(\Omega)}| \leq 2\|u-v_k\|_{L^q(\Omega)}
  \end{align*}
  by H\"{o}lder's inequality.
  Also, since $\nabla u_k=\nabla v_k$ on $\Omega$, it follows that
  \begin{align*}
    \|u-u_k\|_{W^{1,q}(\Omega)} \leq c\|u-v_k\|_{W^{1,q}(\Omega)} \to 0 \quad\text{as}\quad k\to\infty,
  \end{align*}
  i.e., $u_k\to u$ strongly in $W^{1,q}(\Omega)$.
  Hence, $C_{\mathrm{vm}}^\infty(\overline{\Omega})$ is dense in $W_{\mathrm{vm}}^{1,q}(\Omega)$.
\end{proof}

\section{Definition and uniqueness of a weak solution} \label{S:WF_Uni}
In this section, we give the definition of a weak solution to \eqref{E:pLap_MoDo} and prove the uniqueness of a weak solution.
The proof of the existence is presented in the next section.

Fix $p\in(2,\infty)$.
Let us introduce a weak form of \eqref{E:pLap_MoDo}.
Suppose that $u$ and $f$ are smooth and satisfy \eqref{E:pLap_MoDo}.
Then, for a smooth test function $\psi$,
\begin{align*}
  \int_0^T\left(\int_{\Omega_t}\{\partial_tu-\mathrm{div}(|\nabla u|^{p-2}\nabla u)\}\psi\,dx\right)\,dt = \int_0^T\left(\int_{\Omega_t}f\psi\,dx\right)\,dt.
\end{align*}
Moreover, by integration by parts and the boundary condition of \eqref{E:pLap_MoDo},
\begin{align*}
  \int_{\Omega_t}\{\mathrm{div}(|\nabla u|^{p-2}\nabla u)\}\psi\,dx &= \int_{\partial\Omega_t}\{|\nabla u|^{p-2}\partial_\nu u\}\psi\,d\sigma-\int_{\Omega_t}|\nabla u|^{p-2}\nabla u\cdot\nabla\psi\,dx \\
  &= -\int_{\partial\Omega_t}V_\Omega u\psi\,d\sigma-\int_{\Omega_t}|\nabla u|^{p-2}\nabla u\cdot\nabla\psi\,dx.
\end{align*}
We further observe by $V_\Omega=\mathbf{v}_\Omega\cdot\bm{\nu}$ on $\partial\Omega_t$ and the divergence theorem that
\begin{align*}
  \int_{\partial\Omega_t}V_\Omega u\psi\,d\sigma = \int_{\Omega_t}\mathrm{div}(u\psi\mathbf{v}_\Omega)\,dx = \int_{\Omega_t}\{(\mathbf{v}_\Omega\cdot\nabla u)\psi+u(\mathbf{v}_\Omega\cdot\nabla\psi)+u\psi\,\mathrm{div}\,\mathbf{v}_\Omega\}\,dx.
\end{align*}
From the above relations and $\partial^\bullet u=\partial_tu+\mathbf{v}_\Omega\cdot\nabla u$, we deduce that
\begin{multline*}
  \int_0^T\left(\int_{\Omega_t}\{(\partial^\bullet u)\psi+|\nabla u|^{p-2}\nabla u\cdot\nabla\psi+u(\mathbf{v}_\Omega\cdot\nabla\psi)+u\psi\,\mathrm{div}\,\mathbf{v}_\Omega\}\,dx\right)\,dt \\
  = \int_0^T\left(\int_{\Omega_t}f\psi\,dx\right)\,dt.
\end{multline*}
Based on this formula, we define a weak solution to \eqref{E:pLap_MoDo} as follows.

\begin{definition} \label{D:WS_pLap}
  For given $f\in L_{[W^{1,p}]^\ast}^{p'}$ and $u_0\in L^2(\Omega_0)$, we say that $u$ is a weak solution to \eqref{E:pLap_MoDo} if $u\in\mathbb{W}^{p,p'}$ and it satisfies
  \begin{multline} \label{E:pLap_WF}
    \int_0^T\langle\partial^\bullet u(t),\psi(t)\rangle_{W^{1,p}(\Omega_t)}\,dt+\int_0^T\bigl([|\nabla u|^{p-2}\nabla u](t),\nabla\psi(t)\bigr)_{L^2(\Omega_t)}\,dt \\
    +\int_0^T\bigl(u(t),[\mathbf{v}_\Omega\cdot\nabla\psi+\psi\,\mathrm{div}\,\mathbf{v}_\Omega](t)\bigr)_{L^2(\Omega_t)}\,dt = \int_0^T\langle f(t),\psi(t)\rangle_{W^{1,p}(\Omega_t)}\,dt
  \end{multline}
  for all $\psi\in L_{W^{1,p}}^p$ and the initial condition $u(0)=u_0$ in $L^2(\Omega_0)$.
\end{definition}

\begin{remark} \label{R:WS_pLap}
  The above definition makes sense, since $L_{W^{1,p}}^p\subset L_{L^2}^2$ and
  \begin{align*}
    \left|\int_0^T\bigl([|\nabla u|^{p-2}\nabla u](t),\nabla\psi(t)\bigr)_{L^2(\Omega_t)}\,dt\right| \leq \|\nabla u\|_{L_{L^p}^p}^{p-1}\|\nabla\psi\|_{L_{L^p}^p}
  \end{align*}
  by H\"{o}lder's inequality, and since $u(0)\in L^2(\Omega_0)$ by $u\in\mathbb{W}^{p,p'}\subset C_{L^2}$ (see Lemma \ref{L:Trans}).
\end{remark}

The goal of this and the next sections is to establish the following theorem.

\begin{theorem} \label{T:Exi_Uni}
  For all $f\in L_{[W^{1,p}]^\ast}^{p'}$ and $u_0\in L^2(\Omega_0)$, there exists a unique weak solution to \eqref{E:pLap_MoDo}.
\end{theorem}

We first prove the uniqueness of a weak solution.

\begin{proposition} \label{P:Uni}
  Suppose that $u_1$ and $u_2$ are weak solutions to \eqref{E:pLap_MoDo} for the same given data $f\in L_{[W^{1,p}]^\ast}^{p'}$ and $u_0\in L^2(\Omega_0)$.
  Then, $u_1=u_2$ a.e. in $Q_T$.
\end{proposition}

\begin{proof}
  We follow the idea of the proof of \cite[Theorem 2.9]{CaNoOr17}.

  Let $\rho$ be a standard mollifier on $\mathbb{R}$, i.e., for $z\in\mathbb{R}$,
  \begin{align*}
    \rho(z) :=
    \begin{cases}
      c_\rho\exp\left(\dfrac{1}{z^2-1}\right) &(|z|<1), \\
      0 &(|z|\geq1),
    \end{cases}
    \quad c_\rho := \left\{\int_{-1}^1\exp\left(\frac{1}{z^2-1}\right)\,dz\right\}^{-1}.
  \end{align*}
  For $\varepsilon\in(0,1)$, let $\zeta_\varepsilon(z)$ be the mollification of $|z|$ given by
  \begin{align*}
    \zeta_\varepsilon(z) := \frac{1}{\varepsilon}\int_{-\infty}^\infty\rho\left(\frac{z-\tau}{\varepsilon}\right)|\tau|\,d\tau = \int_{-1}^1\rho(s)|z-\varepsilon s|\,ds, \quad z\in\mathbb{R}.
  \end{align*}
  Then, $\zeta_\varepsilon\in C^\infty(\mathbb{R})$.
  For $|z|\leq\varepsilon$, we have
  \begin{align*}
    \zeta_\varepsilon(z) &= \int_{-1}^{z/\varepsilon}\rho(s)(z-\varepsilon s)\,ds+\int_{z/\varepsilon}^1\rho(s)(\varepsilon s-z)\,ds, \\
    \zeta_\varepsilon'(z) &= \int_{-1}^{z/\varepsilon}\rho(s)\,ds-\int_{z/\varepsilon}^1\rho(s)\,ds, \quad \zeta_\varepsilon''(z) = \frac{2}{\varepsilon}\rho\left(\frac{z}{\varepsilon}\right).
  \end{align*}
  Also, $\zeta_\varepsilon(z)=|z|$ for $|z|\geq\varepsilon$.
  Thus, for all $z\in\mathbb{R}$, we easily find that
  \begin{align} \label{Pf_Un:abs}
    \begin{aligned}
      0 &\leq \zeta_\varepsilon(z) \leq 1+|z|, \quad \lim_{\varepsilon\to0}\zeta_\varepsilon(z) = |z|, \\
      -1 &\leq \zeta_\varepsilon'(z) \leq 1, \quad \lim_{\varepsilon\to0}\zeta_\varepsilon'(z) = \mathrm{sgn}\,z =
      \begin{cases}
        z/|z| &(z\neq0), \\
        0 &(z=0),
      \end{cases} \\
      0 &\leq \zeta_\varepsilon''(z) \leq \frac{2c_\rho}{\varepsilon e} \quad (|z| \leq \varepsilon), \quad \zeta_\varepsilon''(z) = 0 \quad (|z| \geq \varepsilon).
    \end{aligned}
  \end{align}
  Now, let $u_1$ and $u_2$ be weak solutions to \eqref{E:pLap_MoDo} with same data $f$ and $u_0$.
  We set
  \begin{align} \label{Pf_Un:vb}
    v := u_1-u_2, \quad \mathbf{b} := |\nabla u_1|^{p-2}\nabla u_1-|\nabla u_2|^{p-2}\nabla u_2 \quad\text{on}\quad Q_T.
  \end{align}
  Then, we see by Definition \ref{D:WS_pLap} that $v\in\mathbb{W}^{p,p'}$, $v(0)=0$ in $L^2(\Omega_0)$, and
  \begin{align} \label{Pf_Un:WF}
    \begin{aligned}
      &\int_0^t\langle\partial^\bullet v(s),\psi(s)\rangle_{W^{1,p}(\Omega_s)}\,ds+\int_0^t\bigl(\mathbf{b}(s),\nabla\psi(s)\bigr)_{L^2(\Omega_s)}\,ds \\
      &\qquad +\int_0^t\bigl(v(s),[\mathbf{v}_\Omega\cdot\nabla\psi+\psi\,\mathrm{div}\,\mathbf{v}_\Omega](s)\bigr)_{L^2(\Omega_s)}\,ds = 0
    \end{aligned}
  \end{align}
  for all $\psi\in L_{W^{1,p}}^p$ and $t\in[0,T]$.
  Also, by Lemma \ref{L:Wpp_Comp}, we have
  \begin{align*}
    \zeta_\varepsilon'(v) \in L_{W^{1,p}}^p, \quad \zeta_\varepsilon(v) \in C_{L^1}.
  \end{align*}
  Thus, we can set $\psi=\zeta_\varepsilon'(v)$ in \eqref{Pf_Un:WF} to get $\sum_{k=1}^4\mathcal{I}_\varepsilon^k(t)=0$, where
  \begin{align*}
    \mathcal{I}_\varepsilon^1(t) &:= \int_0^t\langle\partial^\bullet v(s),[\zeta_\varepsilon'(v)](s)\rangle_{W^{1,p}(\Omega_s)}\,ds, \\
    \mathcal{I}_\varepsilon^2(t) &:= \int_0^t\Bigl(\mathbf{b}(s),\bigl[\nabla[\zeta_\varepsilon'(v)]\bigr](s)\Bigr)_{L^2(\Omega_s)}\,ds, \\
    \mathcal{I}_\varepsilon^3(t) &:= \int_0^t\Bigl(v(s),\bigl[\mathbf{v}_\Omega\cdot\nabla[\zeta_\varepsilon'(v)]\bigr](s)\Bigr)_{L^2(\Omega_s)}\,ds, \\
    \mathcal{I}_\varepsilon^4(t) &:= \int_0^t\bigl(v(s),[\zeta_\varepsilon'(v)\,\mathrm{div}\,\mathbf{v}_\Omega](s)\bigr)_{L^2(\Omega_s)}\,ds.
  \end{align*}
  Let us compute each $\mathcal{I}_\varepsilon^k(t)$.
  First, we see by \eqref{E:Vec_Mono}, \eqref{Pf_Un:abs}, \eqref{Pf_Un:vb} that
  \begin{align} \label{Pf_Un:I2}
    \mathcal{I}_\varepsilon^2(t) = \int_0^t\left(\int_{\Omega_s}[\zeta_\varepsilon''(v)(\mathbf{b}\cdot\nabla v)](s)\,dx\right)\,ds \geq 0.
  \end{align}
  Next, we use \eqref{E:Wpp_Comp} to $\mathcal{I}_\varepsilon^1(t)$.
  Then,
  \begin{align*}
    \mathcal{I}_\varepsilon^1(t) = \int_{\Omega_t}[\zeta_\varepsilon(v)](t)\,dx-\int_{\Omega_0}[\zeta_\varepsilon(v)](0)\,dx-\int_0^t\left(\int_{\Omega_s}[\zeta_\varepsilon(v)\,\mathrm{div}\,\mathbf{v}_\Omega](s)\,dx\right)\,ds.
  \end{align*}
  For each $s\in[0,T]$, we see by \eqref{Pf_Un:abs} that
  \begin{align*}
    0 \leq [\zeta_\varepsilon(v)](s) \leq 1+|v(s)|, \quad \lim_{\varepsilon\to0}[\zeta_\varepsilon(v)](s) = |v(s)| \quad\text{a.e. in}\quad \Omega_s.
  \end{align*}
  Moreover, $1+|v(s)|\in L^1(\Omega_s)$ by $v\in\mathbb{W}^{p,p'}\subset C_{L^2}$ (see Lemma \ref{L:Trans}).
  Thus,
  \begin{align*}
    \lim_{\varepsilon\to0}\int_{\Omega_s}[\zeta_\varepsilon(v)](s)\,dx = \int_{\Omega_s}|v(s)|\,dx \quad\text{for all}\quad s\in[0,T]
  \end{align*}
  by the dominated convergence theorem.
  Similarly, we can show that
  \begin{align*}
    \lim_{\varepsilon\to0}\,\bigl\|\,\zeta_\varepsilon(v)-|v|\,\bigr\|_{L_{L^1}^1} = 0.
  \end{align*}
  From these results and \eqref{E:Vel_Bd}, it follows that
  \begin{align} \label{Pf_Un:I1}
    \lim_{\varepsilon\to0}\mathcal{I}_\varepsilon^1(t) = \int_{\Omega_t}|v(t)|\,dx-\int_{\Omega_0}|v(0)|\,dx-\int_0^t\left(\int_{\Omega_s}\bigl[|v|\,\mathrm{div}\,\mathbf{v}_\Omega\bigr](s)\,dx\right)\,ds.
  \end{align}
  Let us show $\mathcal{I}_\varepsilon^3(t)\to0$ as $\varepsilon\to0$.
  To this end, let
  \begin{align*}
    \mathcal{A}_\varepsilon := \{(x,s)\in Q_T \mid |v(x,s)|\leq \varepsilon\}, \quad \mathcal{A}_0 := \{(x,s)\in Q_T \mid v(x,s) = 0\},
  \end{align*}
  and let $\chi_\varepsilon$ and $\chi_0$ be the characteristic functions of $\mathcal{A}_\varepsilon$ and $\mathcal{A}_0$, respectively.
  Then,
  \begin{align*}
    |\mathcal{I}_\varepsilon^3(t)| &\leq c\int_{Q_T}\bigl[\zeta_\varepsilon''(v)|v||\nabla v|\bigr](x,s)\,dxds \\
    &\leq \frac{c}{\varepsilon}\int_{\mathcal{A}_\varepsilon}\bigl[|v||\nabla v|\bigr](x,s)\,dxds \\
    &\leq c\int_{\mathcal{A}_\varepsilon}|\nabla v(x,s)|\,dxds = c\int_{Q_T}\bigl[\chi_\varepsilon|\nabla v|\bigr](x,s)\,dxds
  \end{align*}
  by $\nabla[\zeta_\varepsilon'(v)]=\zeta_\varepsilon''(v)\nabla v$, \eqref{E:Vel_Bd}, \eqref{Pf_Un:abs}, and the definition of $\mathcal{A}_\varepsilon$.
  Moreover,
  \begin{gather*}
    \chi_\varepsilon|\nabla v| \leq |\nabla v|, \quad \lim_{\varepsilon\to0}\chi_\varepsilon|\nabla v| = \chi_0|\nabla v| = 0 \quad\text{a.e. in}\quad Q_T
  \end{gather*}
  by $\nabla v=0$ a.e. in $\mathcal{A}_0$ (see \cite[Lemma 7.7]{GilTru01}), and
  \begin{align*}
    |\nabla v| \in L_{L^p}^p = L^p(Q_T) \subset L^1(Q_T)
  \end{align*}
  by $p>2$.
  Thus, by the dominated convergence theorem,
  \begin{align} \label{Pf_Un:I3}
    \lim_{\varepsilon\to0}\int_{Q_T}\bigl[\chi_\varepsilon|\nabla v|\bigr](x,s)\,dxds = 0, \quad \lim_{\varepsilon\to0}\mathcal{I}_\varepsilon^3(t) = 0.
  \end{align}
  For $\mathcal{I}_\varepsilon^4(t)$, we see that
  \begin{align*}
    \bigl|\,v\zeta_\varepsilon'(v)-|v|\,\bigr| \leq 2|v|, \quad \lim_{\varepsilon\to0}\bigl(v\zeta_\varepsilon'(v)-|v|\bigr) = v\,\mathrm{sgn}\,v-|v| = 0
  \end{align*}
  a.e. in $Q_T$ by \eqref{Pf_Un:abs}.
  Also, $2|v|\in L_{L^p}^p\subset L^1(Q_T)$ by $p>2$.
  Thus,
  \begin{align*}
    \lim_{\varepsilon\to0}\,\bigl\|\,v\zeta_\varepsilon'(v)-|v|\,\bigr\|_{L_{L^1}^1} = 0
  \end{align*}
  by the dominated convergence theorem.
  By this result and \eqref{E:Vel_Bd}, we obtain
  \begin{align} \label{PF_Un:I4}
    \lim_{\varepsilon\to0}\mathcal{I}_\varepsilon^4(t) = \int_0^t\left(\int_{\Omega_s}\bigl[|v|\,\mathrm{div}\,\mathbf{v}_\Omega\bigr](s)\,dx\right)\,ds.
  \end{align}
  Now, we send $\varepsilon\to0$ in $\sum_{k=1}^4\mathcal{I}_\varepsilon^k(t)=0$ and use \eqref{Pf_Un:I2}--\eqref{PF_Un:I4}.
  Then, we find that
  \begin{align*}
    \int_{\Omega_t}|v(t)|\,dx \leq \int_{\Omega_0}|v(0)|\,dx = 0 \quad\text{for all}\quad t\in[0,T]
  \end{align*}
  by $v(0)=0$ a.e. in $\Omega_0$.
  Therefore, $v=0$, i.e., $u_1=u_2$ a.e. in $Q_T$.
\end{proof}

\section{Existence of a weak solution} \label{S:Exis}
Let us construct a weak solution to \eqref{E:pLap_MoDo} by the Galerkin method.

\subsection{Basis functions} \label{SS:Ex_Basis}
First, we take basis functions on $\Omega_0$.
For a subset $S$ of a linear space (over $\mathbb{R}$), we define the linear span of $S$ by
\begin{align*}
  \mathcal{L}(S) := \{\textstyle\sum_{k=1}^N\alpha_kx_k \mid N\in\mathbb{N}, \, \alpha_1,\dots,\alpha_N\in\mathbb{R}, \, x_1,\dots,x_N\in S\}.
\end{align*}

\begin{lemma} \label{L:Basis}
  There exists a countable subset $\{w_k^0\}_k$ of $C^\infty(\overline{\Omega_0})$ such that
  \begin{itemize}
    \item $w_1^0$ is a nonzero constant function,
    \item $\mathcal{L}(\{w_k^0\}_k)$ is dense in $W^{1,p}(\Omega_0)$, and
    \item $\{w_k^0\}_k$ is an orthonormal basis of $L^2(\Omega_0)$.
  \end{itemize}
\end{lemma}

\begin{proof}
  Let $W_{\mathrm{vm}}^{1,p}(\Omega_0)$ and $C_{\mathrm{vm}}^\infty(\overline{\Omega_0})$ be the function spaces given by \eqref{E:Def_VM}.
  Since $W_{\mathrm{vm}}^{1,p}(\Omega_0)$ can be seen as a subspace of the separable space $L^p(\Omega_0)^{n+1}$ by
  \begin{align*}
    \{(u,\nabla u) \mid u\in W_{\mathrm{vm}}^{1,p}(\Omega_0)\} \subset L^p(\Omega_0)^{n+1},
  \end{align*}
  it follows that $W_{\mathrm{vm}}^{1,p}(\Omega_0)$ is also separable.
  By this fact and Lemmas \ref{L:Separa} and \ref{L:VM_Den}, there exists a countable subset $\mathcal{A}_0$ of $C_{\mathrm{vm}}^\infty(\overline{\Omega_0})$ that is dense in $W_{\mathrm{vm}}^{1,p}(\Omega_0)$.
  We set $\psi_1^0:=1$ on $\Omega_0$ and pick up linearly independent elements $\{\psi_\ell^0\}_{\ell\geq2}$ from $\mathcal{A}_0$ so that $\mathcal{L}(\{\psi_\ell^0\}_{\ell\geq2})$ is again dense in $W_{\mathrm{vm}}^{1,p}(\Omega_0)$.
  Then, splitting $u\in W^{1,p}(\Omega_0)$ into
  \begin{align*}
    u = |\Omega|^{-1}(u,1)_{L^2(\Omega_0)}+u_{\mathrm{vm}}, \quad u_{\mathrm{vm}} := u-|\Omega|^{-1}(u,1)_{L^2(\Omega_0)} \in W_{\mathrm{vm}}^{1,p}(\Omega_0),
  \end{align*}
  we see that $\mathcal{L}(\{\psi_\ell^0\}_{\ell\geq1})$ is dense in $W^{1,p}(\Omega_0)$.
  Since $W^{1,p}(\Omega_0)$ is dense in $L^2(\Omega_0)$ and
  \begin{align*}
    \|u\|_{L^2(\Omega_0)} \leq c\|u\|_{L^p(\Omega_0)} \leq c\|u\|_{W^{1,p}(\Omega_0)}, \quad u \in W^{1,p}(\Omega_0)
  \end{align*}
  by $p>2$ and H\"{o}lder's inequality, it follows that $\mathcal{L}(\{\psi_\ell^0\}_{\ell\geq1})$ is also dense in $L^2(\Omega_0)$.
  We further observe that $\{\psi_\ell^0\}_{\ell\geq1}$ is linearly independent since $\{\psi_\ell^0\}_{\ell\geq2}$ is so and
  \begin{align*}
    (\psi_\ell^0,\psi_1^0)_{L^2(\Omega_0)} = (\psi_\ell^0,1)_{L^2(\Omega_0)} = 0 \quad\text{for all}\quad \ell\geq2
  \end{align*}
  by $\psi_\ell^0\in C_{\mathrm{vm}}^\infty(\overline{\Omega_0})$.
  Hence, we can apply the Gram--Schmidt orthogonalization to $\{\psi_\ell^0\}_{\ell\geq1}$ to get the countable set $\{w_k^0\}_k$ satisfying the conditions of the lemma (note that $w_1^0$ is a nonzero constant function since $\psi_1^0\equiv1$).
\end{proof}

Let $\phi_t$ be the mapping given in Section \ref{SS:Pre_FS}.
For each $k\in\mathbb{N}$ and $t\in[0,T]$, we set
\begin{align*}
  w_k^t := \phi_tw_k^0 \quad\text{on}\quad \overline{\Omega_t}, \quad\text{i.e.,}\quad w_k^t(x) := w_k^0(\Phi_t^{-1}(x)), \quad x\in\overline{\Omega_t}.
\end{align*}
By $w_k^0\in C^\infty(\overline{\Omega_0})$ and Lemma \ref{L:Pinv_Reg}, we see that
\begin{align} \label{E:EBF_Reg}
  w_k^{(\cdot)}, \, \partial_iw_k^{(\cdot)} \in C^1(\overline{Q_T}), \quad i=1,\dots,n.
\end{align}
Moreover, since $w_k^t(\Phi_t(X))=w_k^0(X)$ is independent of $t$ for $(X,t)\in\overline{\Omega_0}\times[0,T]$,
\begin{align} \label{E:MT_EBF}
  \partial^\bullet w_k^t(x) = \frac{\partial}{\partial t}\Bigl(w_k^t(\Phi_t(X))\Bigr)\Big|_{X=\Phi_t^{-1}(x)} = 0, \quad (x,t)\in \overline{Q_T}.
\end{align}
Let $J_t=\det\nabla\Phi_t$ on $\overline{\Omega_0}$ for $t\in[0,T]$.
For $k,\ell\in\mathbb{N}$ and $t\in[0,T]$, we set
\begin{align*}
  M_{k\ell}(t) := \int_{\Omega_t}w_k^t(x)w_\ell^t(x)\,dx = \int_{\Omega_0}w_k^0(X)w_\ell^0(X)J_t(X)\,dX.
\end{align*}
It is continuous on $[0,T]$ by Assumption \ref{A:Domain}.
In general, $M_{k\ell}(t)\neq\delta_{k\ell}$ for $t>0$, where $\delta_{k\ell}$ is the Kronecker delta, but we have the following uniform positive definiteness.

\begin{lemma} \label{L:Uni_PD}
  For $N\in\mathbb{N}$, let $\mathbf{M}_N(t):=\bigl(M_{k\ell}(t)\bigr)_{k,\ell=1,\dots,N}$.
  Then,
  \begin{align} \label{E:Uni_PD}
    \mathbf{a}\cdot[\mathbf{M}_N(t)\mathbf{a}] \geq c_0|\mathbf{a}|^2 \quad\text{for all}\quad (\mathbf{a},t)\in\mathbb{R}^N\times[0,T],
  \end{align}
  where $c_0$ is the constant appearing in \eqref{E:Det_Bd}.
  In particular, $\mathbf{M}_N(t)$ is invertible.
  Also, the inverse matrix $\mathbf{M}_N(t)^{-1}$ is continuous on $[0,T]$.
\end{lemma}

\begin{proof}
  Let $\mathbf{a}=(a_1,\dots,a_N)^{\mathrm{T}}\in\mathbb{R}^N$.
  We set
  \begin{align*}
    u_{\mathbf{a}}(0) := \sum_{k=1}^Na_kw_k^0 \quad\text{on}\quad \Omega_0, \quad u_{\mathbf{a}}(t) := \phi_tu_{\mathbf{a}}(0) = \sum_{k=1}^Na_kw_k^t \quad\text{on}\quad \Omega_t
  \end{align*}
  for $t\in[0,T]$.
  Then, by the change of variables $x=\Phi_t(X)$ and \eqref{E:Det_Bd},
  \begin{align*}
    \mathbf{a}\cdot[\mathbf{M}_N(t)\mathbf{a}] = \sum_{k,\ell=1}^Na_ka_\ell(w_k^t,w_\ell^t)_{L^2(\Omega_t)} = \|u_{\mathbf{a}}(t)\|_{L^2(\Omega_t)}^2 \geq c_0\|u_{\mathbf{a}}(0)\|_{L^2(\Omega_0)}^2 = c_0|\mathbf{a}|^2,
  \end{align*}
  where the last equality holds since $\{w_k^0\}_k$ is orthonormal in $L^2(\Omega_0)$.
  Hence, \eqref{E:Uni_PD} is valid and $\mathbf{M}_N(t)$ is invertible.
  The continuity of $\mathbf{M}_N(t)^{-1}$ follows from that of $\mathbf{M}_N(t)$.
\end{proof}

\subsection{Approximate solutions} \label{SS:Ex_Ap}
Let $f\in L_{[W^{1,p}]^\ast}^{p'}$ and $u_0\in L^2(\Omega_0)$ be given.
We set
\begin{align*}
  u_{0,N} := \sum_{i=1}^N(u_0,w_k^0)_{L^2(\Omega_0)}w_k^0, \quad N\in\mathbb{N}.
\end{align*}
Since $\{w_k^0\}_k$ is an orthonormal basis of $L^2(\Omega_0)$, we have
\begin{align} \label{E:u0N_L2}
  \|u_{0,N}\|_{L^2(\Omega_0)} \leq \|u_0\|_{L^2(\Omega_0)}, \quad \lim_{N\to\infty}\|u_0-u_{0,N}\|_{L^2(\Omega_0)} = 0.
\end{align}
Also, since $F:=\phi_{(\cdot)}^\ast f(\cdot)\in L^{p'}(0,T;[W^{1,p}(\Omega_0)]^\ast)$ with $p'\neq\infty$, we can take
\begin{align} \label{E:FN_DWp}
  F_N \in \mathcal{D}(0,T;[W^{1,p}(\Omega_0)]^\ast) \quad\text{such that}\quad \lim_{N\to\infty}\|F-F_N\|_{L^{p'}(0,T;[W^{1,p}(\Omega_0)]^\ast)} = 0
\end{align}
by cut-off and mollification in $t$.
Let $f_N:=[\phi_{(\cdot)}^\ast]^{-1}F_N\in L_{[W^{1,p}]^\ast}^{p'}$.
Then, by \eqref{E:Dual_Equi},
\begin{align} \label{E:fN_StCo}
  \lim_{N\to\infty}\|f-f_N\|_{L_{[W^{1,p}]^\ast}^{p'}} = 0 \quad\text{and thus}\quad \|f_N\|_{L_{[W^{1,p}]^\ast}^{p'}} \leq c \quad\text{for all}\quad N\in\mathbb{N}.
\end{align}
For each $N\in\mathbb{N}$, we seek for a function $u_N(t)=\sum_{k=1}^N\alpha_N^k(t)w_k^t$ satisfying
\begin{multline} \label{E:Approx}
  \bigl(\partial^\bullet u_N(t),w_k^t\bigr)_{L^2(\Omega_t)}+\bigl([|\nabla u_N|^{p-2}\nabla u_N](t),\nabla w_k^t\bigr)_{L^2(\Omega_t)} \\
  +\bigl(u_N(t),[\mathbf{v}_\Omega\cdot\nabla w_k^t+w_k^t\,\mathrm{div}\,\mathbf{v}_\Omega](t)\bigr)_{L^2(\Omega_t)} = \langle f_N(t),w_k^t\rangle_{W^{1,p}(\Omega_t)}
\end{multline}
for all $t\in(0,T)$ and $k=1,\dots,N$, and the initial condition $u_N(0)=u_{0,N}$.
Let
\begin{align*}
  \gamma_N^k(\mathbf{a},t) &:= \bigl([|\nabla u_{\mathbf{a}}|^{p-2}\nabla u_{\mathbf{a}}](t),\nabla w_k^t\bigr)_{L^2(\Omega_t)}, \quad u_{\mathbf{a}}(t) := \sum_{k=1}^Na_kw_k^t
\end{align*}
for $\mathbf{a}=(a_1,\dots,a_N)^{\mathrm{T}}\in\mathbb{R}^N$, $t\in[0,T]$, and $k=1,\dots,N$.
Also, let
\begin{align} \label{E:Def_fNk}
  f_N^k(t) &:= \langle f_N(t),w_k^t\rangle_{W^{1,p}(\Omega_t)} = \langle \phi_t^\ast f_N(t),w_k^0\rangle_{W^{1,p}(\Omega_0)} = \langle F_N(t),w_k^0\rangle_{W^{1,p}(\Omega_0)}.
\end{align}
We define $\mathbb{R}^N$-valued functions (seen as column vectors)
\begin{align*}
  \bm{\alpha}_N(t) := \bigl(\alpha_N^k(t)\bigr)_k, \quad \bm{\gamma}_N(\mathbf{a},t) := \bigl(\gamma_N^k(\mathbf{a},t)\bigr)_k, \quad \mathbf{f}_N(t) := \bigl(f_N^k(t)\bigr)_k
\end{align*}
and $\mathbb{R}^{N\times N}$-valued functions $\mathbf{M}_N(t)=\bigl(M_{k\ell}(t)\bigr)_{k,\ell=1,\dots,N}$ and
\begin{align} \label{E:Def_BN}
  \mathbf{B}_N(t) = \bigl(B_{k\ell}(t)\bigr)_{k,\ell=1,\dots,N}, \quad B_{k\ell}(t) := \bigl([\mathbf{v}_\Omega\cdot\nabla w_k^t+w_k^t\,\mathrm{div}\,\mathbf{v}_\Omega](t),w_\ell^t\bigr)_{L^2(\Omega_t)}.
\end{align}
Then, by \eqref{E:MT_EBF} and $M_{k\ell}=M_{\ell k}$, we can write \eqref{E:Approx} as
\begin{align*}
  \mathbf{M}_N(t)\frac{d\bm{\alpha}_N}{dt}(t)+\bm{\gamma}_N(\bm{\alpha}_N(t),t)+\mathbf{B}_N(t)\bm{\alpha}_N(t) = \mathbf{f}_N(t).
\end{align*}
Since $\mathbf{M}_N(t)$ is invertible, we find that \eqref{E:Approx} with $u_N(0)=u_{0,N}$ is equivalent to
\begin{align} \label{E:App_ODE}
  \left\{
  \begin{alignedat}{2}
    \frac{d\bm{\alpha}_N}{dt}(t) &= \mathbf{g}_N(\bm{\alpha}_N(t),t), &\quad &\mathbf{g}_N(\mathbf{a},t) := \mathbf{M}_N(t)^{-1}[\mathbf{f}_N(t)-\bm{\gamma}_N(\mathbf{a},t)-\mathbf{B}_N(t)\mathbf{a}], \\
    \alpha_N^k(0) &= (u_0,w_k^0)_{L^2(\Omega_0)}, &\quad &k=1,\dots,N.
  \end{alignedat}
  \right.
\end{align}
By Assumption \ref{A:Domain}, Lemma \ref{L:Uni_PD}, and \eqref{E:FN_DWp}, we see that $\mathbf{g}_N$ is continuous in $t$.
Also, $\bm{\gamma}_N$ and thus $\mathbf{g}_N$ are locally Lipschitz continuous in $\mathbf{a}$ by H\"{o}lder's inequality and
\begin{align*}
  \bigl|\,|\nabla u_{\mathbf{a}}|^{p-2}\nabla u_{\mathbf{a}}-|\nabla u_{\mathbf{b}}|^{p-2}\nabla u_{\mathbf{b}}\,\bigr| \leq c(|\nabla u_{\mathbf{a}}|^{p-2}+|\nabla u_{\mathbf{b}}|^{p-2})|\nabla u_{\mathbf{a}}-\nabla u_{\mathbf{b}}|
\end{align*}
for $\mathbf{a},\mathbf{b}\in\mathbb{R}^N$ due to \eqref{E:p_Lip}.
Thus, by the Cauchy--Lipschitz theorem, we can get a unique $C^1$ solution to \eqref{E:App_ODE} locally in time, which gives a unique local solution to \eqref{E:Approx}.

\subsection{Energy estimate and weak convergence} \label{SS:Ex_Ene}
Next, we show an energy estimate for $u_N$.
In what follows, we write $c_T$ for a general positive constant depending only on $T$.

\begin{proposition} \label{P:Est_Ave}
  For all $t\in[0,T]$ such that $u_N(t)$ exists, we have
  \begin{align}
    |(u_N(t),1)_{L^2(\Omega_t)}| &\leq c_T, \label{E:Est_Ave} \\
    \|u_N(t)\|_{W^{1,p}(\Omega_t)} &\leq c_T\Bigl(\|\nabla u_N(t)\|_{L^p(\Omega_t)}+1\Bigr). \label{E:Est_W1p}
  \end{align}
\end{proposition}

\begin{proof}
  Since $w_1^0$ is a nonzero constant, so is $w_1^t=\phi_tw_1^0=w_1^0$.
  Hence, we have
  \begin{align*}
    \bigl(\partial^\bullet u_N(t),1\bigr)_{L^2(\Omega_t)}+\bigl(u_N(t),\mathrm{div}\,\mathbf{v}_\Omega(t)\bigr)_{L^2(\Omega_t)} = \langle f_N(t),1\rangle_{W^{1,p}(\Omega_t)}
  \end{align*}
  by dividing \eqref{E:Approx} with $k=1$ by $w_1^t$ and using $\nabla w_1^t=0$.
  We integrate both sides over $(0,t)$ and use \eqref{E:Trans}.
  Then, since $\partial^\bullet\psi=0$ for $\psi\equiv1$, we find that
  \begin{align} \label{Pf_EA:equ}
    (u_N(t),1)_{L^2(\Omega_t)} = (u_{0,N},1)_{L^2(\Omega_0)}+\int_0^t\langle f_N(s),1\rangle_{W^{1,p}(\Omega_s)}\,ds.
  \end{align}
  Moreover, by H\"{o}lder's inequality and \eqref{E:u0N_L2}, we see that
  \begin{align*}
    |(u_{0,N},1)_{L^2(\Omega_0)}| \leq |\Omega_0|^{1/2}\|u_{0,N}\|_{L^2(\Omega_0)} \leq |\Omega_0|^{1/2}\|u_0\|_{L^2(\Omega_0)}.
  \end{align*}
  Also, by $\|1\|_{W^{1,p}(\Omega_s)}=\|1\|_{L^p(\Omega_s)}=|\Omega_s|^{1/p}$, \eqref{E:Volume}, H\"{o}lder's inequality, and \eqref{E:fN_StCo},
  \begin{align*}
    \left|\int_0^t\langle f_N(s),1\rangle_{W^{1,p}(\Omega_s)}\,ds\right| &\leq \int_0^T\|f_N(s)\|_{[W^{1,p}(\Omega_s)]^\ast}\|1\|_{W^{1,p}(\Omega_s)}\,ds \\
    &\leq c\int_0^T\|f_N(s)\|_{[W^{1,p}(\Omega_s)]^\ast}\,ds \\
    &\leq cT^{1/p}\|f_N\|_{L_{[W^{1,p}]^\ast}^{p'}} \leq cT^{1/p}.
  \end{align*}
  Applying these estimates to \eqref{Pf_EA:equ}, we get \eqref{E:Est_Ave}.
  Also,
  \begin{align*}
    \|u_N(t)\|_{L^p(\Omega_t)} &\leq c\Bigl(\|\nabla u_N(t)\|_{L^p(\Omega_t)}+|(u_N(t),1)_{L^2(\Omega_t)}|\Bigr) \\
    &\leq c_T\Bigl(\|\nabla u_N(t)\|_{L^p(\Omega_t)}+1\Bigr)
  \end{align*}
  by \eqref{E:Cor_UP} and \eqref{E:Est_Ave}.
  Hence, \eqref{E:Est_W1p} follows.
\end{proof}

\begin{proposition} \label{P:Energy}
  The solution $u_N(t)$ to \eqref{E:Approx} exists on $[0,T]$ and
  \begin{align} \label{E:Energy}
    \|u_N(t)\|_{L^2(\Omega_t)}^2+\int_0^t\|\nabla u_N(s)\|_{L^p(\Omega_s)}^p\,ds \leq c_T \quad\text{for all}\quad t\in[0,T].
  \end{align}
\end{proposition}

\begin{proof}
  We see by \eqref{E:Trans} that
  \begin{align*}
    \bigl(\partial^\bullet u_N(t),u_N(t)\bigr)_{L^2(\Omega_t)} = \frac{1}{2}\frac{d}{dt}\|u_N(t)\|_{L^2(\Omega_t)}^2-\frac{1}{2}\bigl(u_N(t),[u_N\,\mathrm{div}\,\mathbf{v}_\Omega](t)\bigr)_{L^2(\Omega_t)}.
  \end{align*}
  Also, we multiply \eqref{E:Approx} by $\alpha_N^k(t)$ and sum over $k=1,\dots,N$ to get
  \begin{multline*}
    \bigl(\partial^\bullet u_N(t),u_N(t)\bigr)_{L^2(\Omega_t)}+\|\nabla u_N(t)\|_{L^p(\Omega_t)}^p \\
    +\bigl(u_N(t),[\mathbf{v}_\Omega\cdot\nabla u_N+u_N\,\mathrm{div}\,\mathbf{v}_\Omega](t)\bigr)_{L^2(\Omega_t)} = \langle f_N(t),u_N(t)\rangle_{W^{1,p}(\Omega_t)}.
  \end{multline*}
  We combine these relations and integrate over $(0,t)$ to find that
  \begin{align} \label{Pf_Ener:Int}
    \frac{1}{2}\|u_N(t)\|_{L^2(\Omega_t)}^2+\int_0^t\|\nabla u_N(s)\|_{L^p(\Omega_s)}^p\,ds = \frac{1}{2}\|u_{0,N}\|_{L^2(\Omega_0)}^2+\sum_{\ell=1}^3I_\ell(t),
  \end{align}
  where
  \begin{align*}
    I_1(t) &:= \int_0^t\langle f_N(s),u_N(s)\rangle_{W^{1,p}(\Omega_s)}\,ds, \\
    I_2(t) &:= -\int_0^t\bigl(u_N(s),[\mathbf{v}_\Omega\cdot\nabla u_N](s)\bigr)_{L^2(\Omega_s)}\,ds, \\
    I_3(t) &:= -\frac{1}{2}\int_0^t\bigl(u_N(s),[u_N\,\mathrm{div}\,\mathbf{v}_\Omega](s)\bigr)_{L^2(\Omega_s)}\,ds.
  \end{align*}
  To $I_1$, we apply \eqref{E:Est_W1p} and Young's inequality.
  Then, we have
  \begin{align*}
    |I_1(t)| &\leq c_T\int_0^t\|f_N(s)\|_{[W^{1,p}(\Omega_s)]^\ast}\Bigl(\|\nabla u_N(s)\|_{L^p(\Omega_s)}+1\Bigr)\,ds \\
    &\leq \frac{1}{4}\int_0^t\|\nabla u_N(s)\|_{L^p(\Omega_s)}^p\,ds+c_T\int_0^t\Bigl(\|f_N(s)\|_{[W^{1,p}(\Omega_s)]^\ast}^{p'}+1\Bigr)\,ds.
  \end{align*}
  Next, we use \eqref{E:Vel_Bd}, H\"{o}lder's and Young's inequalities, and \eqref{E:Lp_L2} to $I_2$ to get
  \begin{align*}
    |I_2(t)| &\leq c\int_0^t\Bigl(\|u_N(t)\|_{L^2(\Omega_s)}^2+\|\nabla u_N(s)\|_{L^2(\Omega_s)}^2\Bigr)\,ds \\
    &\leq c\delta\int_0^t\|\nabla u_N(s)\|_{L^p(\Omega_s)}^p\,ds+c\int_0^t\Bigl(\|u_N(t)\|_{L^2(\Omega_s)}^2+c_\delta\Bigr)\,ds
  \end{align*}
  for any $\delta>0$.
  In particular, letting $c\delta=1/4$, we obtain
  \begin{align*}
    |I_2(t)| \leq \frac{1}{4}\int_0^t\|\nabla u_N(s)\|_{L^p(\Omega_s)}^p\,ds+c\int_0^t\Bigl(\|u_N(t)\|_{L^2(\Omega_s)}^2+1\Bigr)\,ds.
  \end{align*}
  We also observe by \eqref{E:Vel_Bd} that
  \begin{align*}
    |I_3(t)| \leq c\int_0^t\|u_N(s)\|_{L^2(\Omega_s)}^2\,ds.
  \end{align*}
  Applying these estimates and \eqref{E:u0N_L2} to \eqref{Pf_Ener:Int}, we find that
  \begin{multline*}
    \frac{1}{2}\|u_N(t)\|_{L^2(\Omega_t)}^2+\int_0^t\|\nabla u_N(s)\|_{L^p(\Omega_s)}^p\,ds \\
    \leq \frac{1}{2}\|u_0\|_{L^2(\Omega_0)}^2+\frac{1}{2}\int_0^t\|\nabla u_N(s)\|_{L^p(\Omega_s)}^p\,ds \\ +c_T\int_0^t\Bigl(\|u_N(s)\|_{L^2(\Omega_s)}^2+\|f_N(s)\|_{[W^{1,p}(\Omega_s)]^\ast}^{p'}+1\Bigr)\,ds,
  \end{multline*}
  from which we further deduce that
  \begin{multline*}
    \|u_N(t)\|_{L^2(\Omega_t)}^2+\int_0^t\|\nabla u_N(s)\|_{L^p(\Omega_s)}^p\,ds \\
    \leq \|u_0\|_{L^2(\Omega_0)}^2+c_T\int_0^t\Bigl(\|u_N(s)\|_{L^2(\Omega_s)}^2+\|f_N(s)\|_{[W^{1,p}(\Omega_s)]^\ast}^{p'}+1\Bigr)\,ds.
  \end{multline*}
  Thus, it follows from Gronwall's inequality and
  \begin{align*}
    \int_0^t\Bigl(\|f_N(s)\|_{[W^{1,p}(\Omega_s)]^\ast}^{p'}+1\Bigr)\,ds \leq \|f_N\|_{L_{[W^{1,p}]^\ast}^{p'}}^{p'}+T \leq c+T
  \end{align*}
  by \eqref{E:fN_StCo} that the estimate \eqref{E:Energy} holds as long as $u_N(t)$ exists.
  Then, since
  \begin{align*}
    |\bm{\alpha}_N(t)|^2 \leq c\bm{\alpha}_N(t)\cdot[\mathbf{M}_N(t)\bm{\alpha}_N(t)] = c\|u_N(t)\|_{L^2(\Omega_t)}^2 \leq c_T
  \end{align*}
  by \eqref{E:Uni_PD} and \eqref{E:Energy}, it follows that the solution $\bm{\alpha}_N(t)$ to \eqref{E:App_ODE} is bounded uniformly in $t$, and thus it extends to the whole time interval $[0,T]$.
  Hence, the solution $u_N(t)$ to \eqref{E:Approx} exists on $[0,T]$ and the estimate \eqref{E:Energy} holds for all $t\in[0,T]$.
\end{proof}

Now, we observe by \eqref{E:Est_W1p}, \eqref{E:Energy} and
\begin{align} \label{E:GruN_pd}
  \int_0^T\bigl\|\,[|\nabla u_N|^{p-2}\nabla u_N](t)\,\bigr\|_{L^{p'}(\Omega_t)}^{p'}\,dt = \int_0^T\|\nabla u_N(t)\|_{L^p(\Omega_t)}^p\,dt \leq c_T
\end{align}
that there exist functions $u\in L_{L^2}^\infty\cap L_{W^{1,p}}^p$ and $\mathbf{w}\in[L_{L^{p'}}^{p'}]^n$ such that
\begin{align} \label{E:WeConv}
  \begin{alignedat}{2}
    \lim_{N\to\infty}u_N &= u &\quad &\text{weakly-$\ast$ in $L_{L^2}^\infty$ and weakly in $L_{W^{1,p}}^p$}, \\
    \lim_{N\to\infty}|\nabla u_N|^{p-2}\nabla u_N &= \mathbf{w} &\quad &\text{weakly in $[L_{L^{p'}}^{p'}]^n$}
  \end{alignedat}
\end{align}
up to a subsequence.
It turns out later that $u$ is a weak solution to \eqref{E:pLap_MoDo}, which is unique by Proposition \ref{P:Uni}.
Thus, the convergence \eqref{E:WeConv} is in fact valid for the whole sequence.
We do not repeat this remark in the following discussions.

\subsection{Limit weak forms} \label{SS:Ex_WF}
Let us derive weak forms satisfied by $u$ and $\mathbf{w}$.
In what follows, we sometimes suppress the time variable of functions for the sake of simplicity.

\begin{proposition} \label{P:Lim_WF}
  For all $\psi\in\mathcal{D}_{W^{1,p}}$, we have
  \begin{multline} \label{E:LWF_01}
    -\int_0^T(u,\partial^\bullet\psi)_{L^2(\Omega_t)}\,dt+\int_0^T(\mathbf{w},\nabla\psi)_{L^2(\Omega_t)}\,dt \\
    +\int_0^T(u,\mathbf{v}_\Omega\cdot\nabla\psi)_{L^2(\Omega_t)}\,dt = \int_0^T\langle f,\psi\rangle_{W^{1,p}(\Omega_t)}\,dt.
  \end{multline}
\end{proposition}

\begin{proof}
  Let $k\in\mathbb{N}$ and $\theta\in\mathcal{D}(0,T)$.
  For $N\geq k$, we multiply \eqref{E:Approx} by $\theta(t)$, integrate over $(0,T)$, and use \eqref{E:Trans} and $\theta(0)=\theta(T)=0$.
  Then, we have
  \begin{multline*}
    -\int_0^T\bigl(u_N,\partial^\bullet(\theta w_k^t)\bigr)_{L^2(\Omega_t)}\,dt+\int_0^T\bigl(|\nabla u_N|^{p-2}\nabla u_N,\nabla(\theta w_k^t)\bigr)_{L^2(\Omega_t)}\,dt \\
    +\int_0^T\bigl(u_N,\mathbf{v}_\Omega\cdot\nabla(\theta w_k^t)\bigr)_{L^2(\Omega_t)}\,dt = \int_0^T\langle f_N,\theta w_k^t\rangle_{W^{1,p}(\Omega_t)}\,dt.
  \end{multline*}
  We send $N\to\infty$ and apply \eqref{E:fN_StCo} and \eqref{E:WeConv}.
  Then, we have
  \begin{multline} \label{Pf_LWF:twi}
    -\int_0^T\bigl(u,\partial^\bullet(\theta w_k^t)\bigr)_{L^2(\Omega_t)}\,dt+\int_0^T\bigl(\mathbf{w},\nabla(\theta w_k^t)\bigr)_{L^2(\Omega_t)}\,dt \\
    +\int_0^T\bigl(u,\mathbf{v}_\Omega\cdot\nabla(\theta w_k^t)\bigr)_{L^2(\Omega_t)}\,dt = \int_0^T\langle f,\theta w_k^t\rangle_{W^{1,p}(\Omega_t)}\,dt.
  \end{multline}
  Now, let $\psi\in\mathcal{D}_{W^{1,p}}$.
  By the definitions of $\mathcal{D}_{W^{1,p}}$ and the (strong) material derivative,
  \begin{align*}
    \Psi := \phi_{-(\cdot)}\psi(\cdot) \in \mathcal{D}(0,T;W^{1,p}(\Omega_0)), \quad \partial_t\Psi = \phi_{-(\cdot)}[\partial^\bullet\psi(\cdot)].
  \end{align*}
  Since $\mathcal{L}(\{w_k^0\}_k)$ is dense in $W^{1,p}(\Omega_0)$ (see Lemma \ref{L:Basis}), we can take functions
  \begin{align*}
    \Psi_M(t) = \sum_{\ell=1}^M\theta_M^\ell(t)z_{M,\ell}^0 \quad\text{on}\quad \Omega_0, \quad \theta_M^\ell\in\mathcal{D}(0,T), \quad z_{M,\ell}^0\in\mathcal{L}(\{w_k^0\}_k)
  \end{align*}
  for $M\in\mathbb{N}$ such that
  \begin{align} \label{Pf_LWF:Psi}
    \lim_{M\to\infty}\|\Psi-\Psi_M\|_{L^p(0,T;W^{1,p}(\Omega_0))} = 0, \quad \lim_{K\to\infty}\|\partial_t\Psi-\partial_t\Psi_M\|_{L^2(0,T;W^{1,p}(\Omega_0))} = 0.
  \end{align}
  This can be shown as in \cite[Lemma 2.2]{Mas84}, so we omit the proof here.
  Let
  \begin{align*}
    \psi_M(t) := \phi_t\Psi_M(t) = \sum_{\ell=1}^M\theta_M^\ell(t)z_{M,\ell}^t \quad\text{on}\quad \Omega_t, \quad z_{M,\ell}^t := \phi_tz_{M,\ell}^0 \in \mathcal{L}(\{w_k^t\}_k).
  \end{align*}
  Then, by $\partial_t\Psi_M=\phi_{-(\cdot)}[\partial^\bullet\psi_M(\cdot)]$, \eqref{E:LBq_Equi}, and \eqref{Pf_LWF:Psi}, we have
  \begin{align*}
    \lim_{K\to\infty}\|\psi-\psi_M\|_{L_{W^{1,p}}^p} = 0, \quad \lim_{K\to\infty}\|\partial^\bullet\psi-\partial^\bullet\psi_M\|_{L_{W^{1,p}}^2} = 0.
  \end{align*}
  We see that \eqref{E:LWF_01} holds for each $\psi_M$, since \eqref{Pf_LWF:twi} is valid for all $k\in\mathbb{N}$ and $\theta\in\mathcal{D}(0,T)$.
  Letting $M\to\infty$ and using the above convergence, we get \eqref{E:LWF_01} for $\psi$.
\end{proof}

\begin{proposition} \label{P:Lim_Wpp}
  We have $u\in\mathbb{W}^{p,p'}$ and
  \begin{multline} \label{E:LWF_02}
    \int_0^T\langle \partial^\bullet u,\psi\rangle_{W^{1,p}(\Omega_t)}\,dt+\int_0^T(\mathbf{w},\nabla\psi)_{L^2(\Omega_t)}\,dt \\
    +\int_0^T\bigl(u,\mathbf{v}_\Omega\cdot\nabla\psi+\psi\,\mathrm{div}\,\mathbf{v}_\Omega)_{L^2(\Omega_t)}\,dt = \int_0^T\langle f,\psi\rangle_{W^{1,p}(\Omega_t)}\,dt
  \end{multline}
  for all $\psi\in L_{W^{1,p}}^p$.
  Moreover, $u(0)=u_0$ in $L^2(\Omega_0)$.
\end{proposition}

\begin{proof}
  Let $\psi\in\mathcal{D}_{W^{1,p}}$.
  We subtract $\int_0^T(u,\psi\,\mathrm{div}\,\mathbf{v}_\Omega)_{L^2(\Omega_t)}\,dt$ from both sides of \eqref{E:LWF_01} and rewrite the resulting equality as
  \begin{align} \label{Pf_LWp:Dual}
    -\int_0^T(u,\partial^\bullet\psi+\psi\,\mathrm{div}\,\mathbf{v}_\Omega)_{L^2(\Omega_t)}\,dt = L_{u,\mathbf{w}}(\psi),
  \end{align}
  where $L_{u,\mathbf{w}}$ is a linear functional given by
  \begin{multline*}
    L_{u,\mathbf{w}}(\psi) := \int_0^T\langle f,\psi\rangle_{W^{1,p}(\Omega_t)}\,dt-\int_0^T(\mathbf{w},\nabla\psi)_{L^2(\Omega_t)}\,dt \\
    -\int_0^T(u,\mathbf{v}_\Omega\cdot\nabla\psi+\psi\,\mathrm{div}\,\mathbf{v}_\Omega)_{L^2(\Omega_t)}\,dt.
  \end{multline*}
  We observe by \eqref{E:Vel_Bd} and H\"{o}lder's inequality that
  \begin{align*}
    |L_{u,\mathbf{w}}(\psi)| \leq \|f\|_{L_{[W^{1,p}]^\ast}^{p'}}\|\psi\|_{L_{W^{1,p}}^p}+\|\mathbf{w}\|_{L_{L^{p'}}^{p'}}\|\nabla\psi\|_{L_{L^p}^p}+\|u\|_{L_{L^2}^2}\Bigl(\|\psi\|_{L_{L^2}^2}+\|\nabla\psi\|_{L_{L^2}^2}\Bigr).
  \end{align*}
  Also, by $p>2$, H\"{o}lder's inequality, \eqref{E:Volume}, and $T<\infty$, we have
  \begin{align*}
    \|\psi\|_{L_{L^2}^2}+\|\nabla\psi\|_{L_{L^2}^2} \leq c_T\|\psi\|_{L_{W^{1,p}}^p}, \quad \|\nabla\psi\|_{L_{L^p}^p} \leq \|\psi\|_{L_{W^{1,p}}^p}.
  \end{align*}
  Thus, $L_{u,\mathbf{w}}$ extends to a bounded linear functional on $L_{W^{1,p}}^p$ by the density of $\mathcal{D}_{W^{1,p}}$ in $L_{W^{1,p}}^p$.
  In view of $[L_{W^{1,p}}^p]^\ast=L_{[W^{1,p}]^\ast}^{p'}$, \eqref{Pf_LWp:Dual}, and Definition \ref{D:We_MatDe}, this means that
  \begin{align*}
    \partial^\bullet u\in L_{[W^{1,p}]^\ast}^{p'}, \quad  \int_0^T\langle\partial^\bullet u,\psi\rangle_{W^{1,p}(\Omega_t)}\,dt = L_{u,\mathbf{w}}(\psi) \quad\text{for all}\quad \psi\in\mathcal{D}_{W^{1,p}}.
  \end{align*}
  Thus, $u\in\mathbb{W}^{p,p'}$, and \eqref{E:LWF_02} holds for all $\psi\in L_{W^{1,p}}^p$ by the density of $\mathcal{D}_{W^{1,p}}$ in $L_{W^{1,p}}^p$.

  Let us show $u(0)=u_0$.
  Fix any $\Psi_0\in\mathcal{D}(\Omega_0)$.
  Let $\Psi_0^\ell:=(\Psi_0,w_\ell^0)_{L^2(\Omega_0)}$ and
  \begin{align*}
    \Psi_{0,M} := \sum_{\ell=1}^M\Psi_0^\ell w_\ell^0 \quad\text{on}\quad \Omega_0, \quad \lim_{M\to\infty}\|\Psi_0-\Psi_{0,M}\|_{L^2(\Omega_0)} = 0.
  \end{align*}
  Also, let $\theta\in C^1([0,T])$ satisfy $\theta(0)=1$ and $\theta(T)=0$, and let
  \begin{align*}
    \psi_M(t) := \theta(t)[\phi_t\Psi_{0,M}] = \theta(t)\sum_{\ell=1}^M\Psi_0^\ell w_\ell^t \quad\text{on}\quad \Omega_t, \quad t\in[0,T], \quad M\in\mathbb{N}.
  \end{align*}
  Let $N\geq M$.
  We multiply \eqref{E:Approx} by $\theta(t)\Psi_0^k$ and sum over $k=1,\dots,M$.
  Then,
  \begin{multline*}
    (\partial^\bullet u_N,\psi_M)_{L^2(\Omega_t)}+(|\nabla u_N|^{p-2}\nabla u_N,\nabla\psi_M)_{L^2(\Omega_t)} \\
    +(u_N,\mathbf{v}_\Omega\cdot\nabla\psi_M+\psi_M\,\mathrm{div}\,\mathbf{v}_\Omega)_{L^2(\Omega_t)} = \langle f_N,\psi_M\rangle_{W^{1,p}(\Omega_t)}.
  \end{multline*}
  We integrate both sides over $(0,T)$ and use \eqref{E:Trans}, $\theta(0)=1$, and $\theta(T)=0$ to get
  \begin{multline*}
    -(u_{0,N},\Psi_{0,M})_{L^2(\Omega_0)}-\int_0^T(u_N,\partial^\bullet\psi_M)_{L^2(\Omega_t)}\,dt+\int_0^T(|\nabla u_N|^{p-2}\nabla u_N,\nabla\psi_M)_{L^2(\Omega_t)}\,dt \\
    +\int_0^T(u_N,\mathbf{v}_\Omega\cdot\nabla\psi_M)_{L^2(\Omega_t)}\,dt = \int_0^T\langle f_N,\psi_M\rangle_{W^{1,p}(\Omega_t)}\,dt.
  \end{multline*}
  Let $N\to\infty$ in the above equality.
  Then, by \eqref{E:u0N_L2}, \eqref{E:fN_StCo}, and \eqref{E:WeConv}, we have
  \begin{multline*}
    -(u_0,\Psi_{0,M})_{L^2(\Omega_0)}-\int_0^T(u,\partial^\bullet\psi_M)_{L^2(\Omega_t)}\,dt+\int_0^T(\mathbf{w},\nabla\psi_M)_{L^2(\Omega_t)}\,dt \\
    +\int_0^T(u,\mathbf{v}_\Omega\cdot\nabla\psi_M)_{L^2(\Omega_t)}\,dt = \int_0^T\langle f,\psi_M\rangle_{W^{1,p}(\Omega_t)}\,dt.
  \end{multline*}
  On the other hand, setting $\psi=\psi_M$ in \eqref{E:LWF_02} and applying \eqref{E:Trans}, we find that
  \begin{multline*}
    -(u(0),\Psi_{0,M})_{L^2(\Omega_0)}-\int_0^T(u,\partial^\bullet\psi_M)_{L^2(\Omega_t)}\,dt+\int_0^T(\mathbf{w},\nabla\psi_M)_{L^2(\Omega_t)}\,dt \\
    +\int_0^T(u,\mathbf{v}_\Omega\cdot\nabla\psi_M)_{L^2(\Omega_t)}\,dt = \int_0^T\langle f,\psi_M\rangle_{W^{1,p}(\Omega_t)}\,dt
  \end{multline*}
  by $\theta(0)=1$ and $\theta(T)=0$.
  From the above relations, we deduce that
  \begin{align*}
    (u(0),\Psi_{0,M})_{L^2(\Omega_0)} = (u_0,\Psi_{0,M})_{L^2(\Omega_0)} \quad\text{for all}\quad M\in\mathbb{N}.
  \end{align*}
  Thus, letting $M\to\infty$ and using $\Psi_{0,M}\to\Psi_0$ strongly in $L^2(\Omega_0)$, we obtain
  \begin{align*}
    (u(0),\Psi_0)_{L^2(\Omega_0)} = (u_0,\Psi_0)_{L^2(\Omega_0)} \quad\text{for all}\quad \Psi_0\in\mathcal{D}(\Omega_0),
  \end{align*}
  and we conclude that $u(0)=u_0$ in $L^2(\Omega_0)$ since $\mathcal{D}(\Omega_0)$ is dense in $L^2(\Omega_0)$.
\end{proof}

\subsection{Strong convergence} \label{SS:Ex_ST}
The remaining task is to identify $\mathbf{w}$ in terms of $u$.
To this end, we use a monotonicity argument as in the case of non-moving domains (see \cite{LaSoUr68,Lio69,Zei90_2B}), but we require the strong convergence of $u_N$ in the present case.
We get the strong convergence below by adopting the idea used in \cite[Chapter V, Theorem 6.7]{LaSoUr68} with modifications suited for moving domains.

\begin{proposition} \label{P:uN_AsAr}
  For each fixed $k\in\mathbb{N}$, the sequence $\{(u_N(t),w_k^t)_{L^2(\Omega_t)}\}_{N=k}^\infty$ is uniformly bounded and equicontinuous on $[0,T]$.
\end{proposition}

\begin{proof}
  First, we see by $w_k^t=w_k^0\circ\Phi_t^{-1}$ on $\Omega_t$, \eqref{E:Det_Bd}, and \eqref{E:Grad_Bd} that
  \begin{align} \label{Pf_AsAr:wkt}
    \|w_k^t\|_{L^r(\Omega_t)} \leq c\|w_k^0\|_{L^r(\Omega_0)}, \quad \|\nabla w_k^t\|_{L^r(\Omega_t)} \leq c\|\nabla w_k^0\|_{L^r(\Omega_0)}
  \end{align}
  for all $t\in[0,T]$ and $r=2,p$.
  Thus, for $0\leq s\leq t\leq T$, we have
  \begin{align} \label{Pf_AsAr:Int}
    \begin{aligned}
      \int_s^t\|w_k^\tau\|_{W^{1,p}(\Omega_\tau)}^p\,d\tau &\leq c\int_s^t\|w_k^0\|_{W^{1,p}(\Omega_0)}^p\,d\tau = c(t-s)\|w_k^0\|_{W^{1,p}(\Omega_0)}^p, \\
      \int_s^t\|\nabla w_k^\tau\|_{L^2(\Omega_\tau)}\,d\tau &\leq c\int_s^t\|\nabla w_k^0\|_{L^2(\Omega_0)}\,d\tau = c(t-s)\|\nabla w_k^0\|_{L^2(\Omega_0)},
    \end{aligned}
  \end{align}
  since $\|w_k^0\|_{W^{1,p}(\Omega_0)}$ and $\|\nabla w_k^0\|_{L^2(\Omega_0)}$ are independent of $\tau$.

  We observe by H\"{o}lder's inequality, \eqref{E:Energy}, and \eqref{Pf_AsAr:wkt} that
  \begin{align*}
    \max_{t\in[0,T]}|(u_N(t),w_k^t)_{L^2(\Omega_t)}| \leq \max_{t\in[0,T]}\|u_N(t)\|_{L^2(\Omega_t)}\|w_k^t\|_{L^2(\Omega_t)} \leq c_T\|w_k^0\|_{L^2(\Omega_0)}.
  \end{align*}
  Hence, $\{(u_N(t),w_k^t)_{L^2(\Omega_t)}\}_{N=k}^\infty$ is uniformly bounded on $[0,T]$.

  Next, we see by \eqref{E:Trans} and $\partial^\bullet w_k^t=0$ on $\Omega_t$ (see \eqref{E:MT_EBF}) that
  \begin{align*}
    &(u_N(t),w_k^t)_{L^2(\Omega_t)}-(u_N(s),w_k^s)_{L^2(\Omega_s)} = \int_s^t\frac{d}{d\tau}(u_N(\tau),w_k^\tau)_{L^2(\Omega_\tau)}\,d\tau \\
    &\qquad = \int_s^t\Bigl\{\bigl(\partial^\bullet u_N(\tau),w_k^\tau\bigr)_{L^2(\Omega_\tau)}+\bigl(u_N(\tau),[w_k^\tau\,\mathrm{div}\,\mathbf{v}_\Omega](\tau)\bigr)_{L^2(\Omega_\tau)}\Bigr\}\,d\tau
  \end{align*}
  for $0\leq s\leq t\leq T$.
  We further use \eqref{E:Approx} to the second line to get
  \begin{align} \label{Pf_AsAr:Diff}
    (u_N(t),w_k^t)_{L^2(\Omega_t)}-(u_N(s),w_k^s)_{L^2(\Omega_s)} = \sum_{\ell=1}^3K_\ell(s,t),
  \end{align}
  where
  \begin{align*}
    K_1(s,t) &:= \int_s^t\langle f_N(\tau),w_k^\tau\rangle_{W^{1,p}(\Omega_\tau)}\,d\tau, \\
    K_2(s,t) &:= -\int_s^t\bigl([|\nabla u_N|^{p-2}\nabla u_N](\tau),\nabla w_k^\tau\bigr)_{L^2(\Omega_\tau)}\,d\tau, \\
    K_3(s,t) &:= -\int_s^t\bigl(u_N(\tau),[\mathbf{v}_\Omega\cdot\nabla w_k^\tau](\tau)\bigr)_{L^2(\Omega_\tau)}\,d\tau.
  \end{align*}
  For $K_1$ and $K_2$, we see by H\"{o}lder's inequality that
  \begin{align*}
    |K_1(s,t)| &\leq \|f_N\|_{L_{[W^{1,p}]^\ast}^{p'}}\left(\int_s^t\|w_k^\tau\|_{W^{1,p}(\Omega_\tau)}^p\,d\tau\right)^{1/p}, \\
    |K_2(s,t)| &\leq \bigl\|\,|\nabla u_N|^{p-2}\nabla u_N\|_{L_{L^{p'}}^{p'}}\left(\int_s^t\|\nabla w_k^\tau\|_{L^p(\Omega_\tau)}^p\,d\tau\right)^{1/p}.
  \end{align*}
  Thus, it follows from \eqref{E:fN_StCo}, \eqref{E:GruN_pd}, and \eqref{Pf_AsAr:Int} that
  \begin{align*}
    |K_\ell(s,t)| \leq c_T(t-s)^{1/p}\|w_k^0\|_{W^{1,p}(\Omega_0)}, \quad \ell = 1,2.
  \end{align*}
  Also, to $K_3$, we apply \eqref{E:Vel_Bd}, H\"{o}lder's inequality, \eqref{E:Energy}, and \eqref{Pf_AsAr:Int}.
  Then,
  \begin{align*}
    |K_3(s,t)| \leq c\int_s^t\|u_N(\tau)\|_{L^2(\Omega_\tau)}\|\nabla w_k^\tau\|_{L^2(\Omega_\tau)}\,d\tau \leq c_T(t-s)\|\nabla w_k^0\|_{L^2(\Omega_0)}.
  \end{align*}
  Applying these estimates to \eqref{Pf_AsAr:Diff}, we find that
  \begin{multline*}
    |(u_N(t),w_k^t)_{L^2(\Omega_t)}-(u_N(s),w_k^s)_{L^2(\Omega_s)}| \\
    \leq c_T\Bigl\{(t-s)^{1/p}\|w_k^0\|_{W^{1,p}(\Omega_0)}+(t-s)\|\nabla w_k^0\|_{L^2(\Omega_0)}\Bigr\}.
  \end{multline*}
  Hence, $\{(u_N(t),w_k^t)_{L^2(\Omega_t)}\}_{N=k}^\infty$ is equicontinuous on $[0,T]$.
\end{proof}

Next, we prove a uniform-in-time Friedrichs type inequality on $\Omega_t$.

\begin{lemma} \label{L:Fried}
  There exists a constant $c>0$ depending only on the constants $c_0$, $c_1$, and $c_2$ appearing in \eqref{E:Det_Bd} and \eqref{E:Grad_Bd} such that the following statement holds: for each $\varepsilon>0$, there exists a $K_\varepsilon\in\mathbb{N}$ such that
  \begin{align} \label{E:Fried}
    \|\psi\|_{L^2(\Omega_t)}^2 \leq c\left(\sum_{k=1}^{K_\varepsilon}|(\psi,w_k^t)_{L^2(\Omega_t)}|^2+\varepsilon\|\psi\|_{H^1(\Omega_t)}^2\right)
  \end{align}
  for all $t\in[0,T]$ and $\psi\in H^1(\Omega_t)$ (note that $K_\varepsilon$ is independent of $t$ and $\psi$).
\end{lemma}

\begin{proof}
  First, let $t=0$.
  Then, since $\{w_k^0\}_k$ is an orthonormal basis of $L^2(\Omega_0)$, we have the following Friedrichs inequality on $\Omega_0$: for each $\varepsilon>0$, there exists a $K_\varepsilon\in\mathbb{N}$ such that
  \begin{align} \label{Pf_Fri:zero}
    \|Z\|_{L^2(\Omega_0)}^2 \leq \sum_{k=1}^{K_\varepsilon}|(Z,w_k^0)_{L^2(\Omega_0)}|^2+\varepsilon\|Z\|_{H^1(\Omega_0)}^2
  \end{align}
  for all $Z\in H^1(\Omega_0)$.
  We refer to \cite[Chapter II, Lemma 2.4]{LaSoUr68} for the proof.

  Now, let $t\in[0,T]$ and $\psi\in H^1(\Omega_t)$.
  We define
  \begin{align*}
    Z(X) := \psi(\Phi_t(X))J_t(X) = \psi(\Phi_t(X))\det\nabla\Phi_t(X), \quad X\in\Omega_0.
  \end{align*}
  Then, we observe by \eqref{E:Det_Bd} and \eqref{E:Grad_Bd} that
  \begin{align*}
    \|Z\|_{L^2(\Omega_0)} \geq c\|\psi\|_{L^2(\Omega_t)}, \quad \|Z\|_{H^1(\Omega_0)} \leq c\|\psi\|_{H^1(\Omega_t)},
  \end{align*}
  where $c>0$ is a constant depending only on $c_0$, $c_1$, and $c_2$.
  Moreover,
  \begin{align*}
    (Z,w_k^0)_{L^2(\Omega_0)} = \int_{\Omega_0}\psi(\Phi_t(X))w_k^0(X)J_t(X)\,dX = \int_{\Omega_t}\psi(x)w_k^t(x)\,dx = (\psi,w_k^t)_{L^2(\Omega_t)}
  \end{align*}
  by $w_k^t=w_k^0\circ\Phi_t^{-1}$ on $\Omega_t$.
  Applying these relations to \eqref{Pf_Fri:zero}, we get \eqref{E:Fried}.
\end{proof}

The above results give the strong convergence of $u_N$ in $L_{L^2}^2$ as follows.

\begin{proposition} \label{P:uN_Str}
  Up to a subsequence, $u_N\to u$ strongly in $L_{L^2}^2$ as $N\to\infty$.
\end{proposition}

\begin{proof}
  By Proposition \ref{P:uN_AsAr}, the Ascoli--Arzel\'{a} theorem, and a diagonal argument, we can take a subsequence of $\{u_N\}_N$, which is still denoted by $\{u_N\}_N$, such that
  \begin{align} \label{Pf_uSt:Unif}
    \{(u_N(t),w_k^t)_{L^2(\Omega_t)}\}_{N=1}^\infty \quad\text{converges uniformly on $[0,T]$ for all $k\in\mathbb{N}$}.
  \end{align}
  Also, we see by \eqref{E:Lp_L2} with $\delta=1$ and by \eqref{E:Energy} and $T<\infty$ that
  \begin{align} \label{Pf_uSt:H1}
    \int_0^T\|u_N(t)\|_{H^1(\Omega_t)}^2\,dt \leq c\int_0^T\Bigl(\|u_N(t)\|_{L^2(\Omega_t)}^2+\|\nabla u_N(t)\|_{L^p(\Omega_t)}^p+1\Bigr)\,dt \leq c_T.
  \end{align}
  For each $\varepsilon>0$, let $K_\varepsilon\in\mathbb{N}$ be given in Lemma \ref{L:Fried}.
  Noting that
  \begin{align*}
    \|u_M(t)-u_N(t)\|_{H^1(\Omega_t)}^2 \leq 2\Bigl(\|u_M(t)\|_{H^1(\Omega_t)}^2+\|u_N(t)\|_{H^1(\Omega_t)}^2\Bigr),
  \end{align*}
  we set $\psi=u_M-u_N$ in \eqref{E:Fried}, integrate over $(0,T)$, and use \eqref{Pf_uSt:H1} to find that
  \begin{align*}
    \|u_M-u_N\|_{L_{L^2}^2}^2 \leq c\left(\sum_{k=1}^{K_\varepsilon}\int_0^T\bigl|\bigl(u_M(t)-u_N(t),w_k^t\bigr)_{L^2(\Omega_t)}\bigr|^2\,dt+\varepsilon c_T\right).
  \end{align*}
  Let $M,N\to\infty$.
  Then, since \eqref{Pf_uSt:Unif} holds and $K_\varepsilon$ is independent of $M$ and $N$,
  \begin{align*}
    \limsup_{M,N\to\infty}\|u_M-u_N\|_{L_{L^2}^2}^2 \leq c_T\varepsilon \quad\text{for all}\quad \varepsilon>0.
  \end{align*}
  Thus, $\{u_N\}_N$ is a Cauchy sequence in $L_{L^2}^2$.
  By this fact and \eqref{E:WeConv}, we find that $\{u_N\}_N$ converges to $u$ strongly in $L_{L^2}^2$.
\end{proof}

\subsection{Characterization of the nonlinear term} \label{SS:Ex_Char}
Now, let us identify $\mathbf{w}$.
As mentioned in Section \ref{S:Intro}, the strong convergence of $u_N$ enables us not only to deal with
\begin{align*}
  \int_0^T\bigl(u_N(t),[u_N\,\mathrm{div}\,\mathbf{v}_\Omega](t)\bigr)_{L^2(\Omega_t)}\,dt, \quad \int_0^T\bigl(u_N(t),[\mathbf{v}_\Omega\cdot\nabla u_N](t)\bigr)_{L^2(\Omega_t)}\,dt,
\end{align*}
but also to avoid considering the time traces $u_N(0)$ and $u_N(T)$ that usually appear when one uses a monotonicity argument (see e.g. \cite{LaSoUr68,Lio69,Zei90_2B,AlCaDjEl23}).

\begin{proposition} \label{P:Chara}
  For all $\psi\in L_{W^{1,p}}^p$, we have
  \begin{align} \label{E:Chara}
    \int_0^T(\mathbf{w},\nabla\psi)_{L^2(\Omega_t)}\,dt = \int_0^T(|\nabla u|^{p-2}\nabla u,\nabla\psi)_{L^2(\Omega_t)}\,dt.
  \end{align}
\end{proposition}

\begin{proof}
  Let $\zeta\in L_{W^{1,p}}^p$, and let $\theta\in\mathcal{D}(0,T)$ satisfy $\theta\geq0$ on $(0,T)$.
  For $N\in\mathbb{N}$ and $t\in[0,T]$, we set $\mathbf{a}=\nabla u_N(t)$ and $\mathbf{b}=\nabla \zeta(t)$ in \eqref{E:Vec_Mono} and multiply both sides by $\theta(t)\geq0$.
  Then,
  \begin{align*}
    (|\nabla u_N|^{p-2}\nabla u_N-|\nabla\zeta|^{p-2}\nabla\zeta,\theta\nabla u_N-\theta\nabla\zeta)_{L^2(\Omega_t)} \geq 0.
  \end{align*}
  Next, we multiply \eqref{E:Approx} by $\theta(t)\alpha_N^k(t)$ and sum over $k=1,\dots,N$ to get
  \begin{multline*}
    (\partial^\bullet u_N,\theta u_N)_{L^2(\Omega_t)}+(|\nabla u_N|^{p-2}\nabla u_N,\theta\nabla u_N)_{L^2(\Omega_t)} \\
    +(u_N,\theta[\mathbf{v}_\Omega\cdot\nabla u_N+u_N\,\mathrm{div}\,\mathbf{v}_\Omega])_{L^2(\Omega_t)} = \langle f_N,\theta u_N\rangle_{W^{1,p}(\Omega_t)}.
  \end{multline*}
  From the above two relations, we deduce that
  \begin{multline*}
    -(\partial^\bullet u_N,\theta u_N)_{L^2(\Omega_t)}-(u_N,\theta[\mathbf{v}_\Omega\cdot\nabla u_N+u_N\,\mathrm{div}\,\mathbf{v}_\Omega])_{L^2(\Omega_t)}+\langle f_N,\theta u_N\rangle_{W^{1,p}(\Omega_t)}\\
    -(|\nabla u_N|^{p-2}\nabla u_N,\theta\nabla\zeta)_{L^2(\Omega_t)}-(|\nabla\zeta|^{p-2}\nabla\zeta,\theta\nabla u_N-\theta\nabla\zeta)_{L^2(\Omega_t)} \geq 0.
  \end{multline*}
  We integrate both sides over $(0,T)$ and use \eqref{E:Tr_Mult}.
  Then, noting that $\theta$ depends only on time and $\theta(0)=\theta(T)=0$, and writing $\theta'=d\theta/dt$, we have
  \begin{multline} \label{Pf_Ch:uN}
    \frac{1}{2}\int_0^T(u_N,\theta'u_N)_{L^2(\Omega_t)}\,dt-\int_0^T\Bigl(u_N,\theta\Bigl[\mathbf{v}_\Omega\cdot\nabla u_N+\frac{1}{2}u_N\,\mathrm{div}\,\mathbf{v}_\Omega\Bigr]\Bigr)_{L^2(\Omega_t)}\,dt \\
    +\int_0^T\langle f_N,\theta u_N\rangle_{W^{1,p}(\Omega_t)}\,dt-\int_0^T(|\nabla u_N|^{p-2}\nabla u_N,\theta\nabla\zeta)_{L^2(\Omega_t)}\,dt \\
    -\int_0^T(|\nabla\zeta|^{p-2}\nabla\zeta,\theta\nabla u_N-\theta\nabla\zeta)_{L^2(\Omega_t)}\,dt \geq 0.
  \end{multline}
  Now, we see by Proposition \ref{P:uN_Str}, $p'<2$, H\"{o}lder's inequality, \eqref{E:Volume}, and $T<\infty$ that
  \begin{align*}
    \lim_{N\to\infty}u_N = u \quad\text{strongly in $L_{L^2}^2$ and thus in $L_{L^{p'}}^{p'}$}.
  \end{align*}
  Using this, \eqref{E:fN_StCo}, \eqref{E:WeConv}, $\theta\in\mathcal{D}(0,T)$, and \eqref{E:Vel_Bd}, we send $N\to\infty$ in \eqref{Pf_Ch:uN} to get
  \begin{multline*}
    \frac{1}{2}\int_0^T(u,\theta'u)_{L^2(\Omega_t)}\,dt-\int_0^T\Bigl(u,\theta\Bigl[\mathbf{v}_\Omega\cdot\nabla u+\frac{1}{2}u\,\mathrm{div}\,\mathbf{v}_\Omega\Bigr]\Bigr)_{L^2(\Omega_t)}\,dt \\
    +\int_0^T\langle f,\theta u\rangle_{W^{1,p}(\Omega_t)}\,dt-\int_0^T(\mathbf{w},\theta\nabla\zeta)_{L^2(\Omega_t)}\,dt \\
    -\int_0^T(|\nabla\zeta|^{p-2}\nabla\zeta,\theta\nabla u-\theta\nabla\zeta)_{L^2(\Omega_t)}\,dt \geq 0.
  \end{multline*}
  On the other hand, we set $\psi=\theta u$ in \eqref{E:LWF_02} and use \eqref{E:Tr_Mult} and $\theta(0)=\theta(T)=0$.
  Then,
  \begin{multline*}
    -\frac{1}{2}\int_0^T(u,\theta' u)_{L^2(\Omega_t)}\,dt+\int_0^T(\mathbf{w},\theta\nabla u)_{L^2(\Omega_t)}\,dt \\
    +\int_0^T\Bigl(u,\theta\Bigl[\mathbf{v}_\Omega\cdot\nabla u+\frac{1}{2}u\,\mathrm{div}\,\mathbf{v}_\Omega\Bigr]\Bigr)_{L^2(\Omega_t)}\,dt = \int_0^T\langle f,\theta u\rangle_{W^{1,p}(\Omega_t)}\,dt.
  \end{multline*}
  Taking the sum of the above two relations, we find that
  \begin{align*}
    \int_0^T(\mathbf{w}-|\nabla\zeta|^{p-2}\nabla\zeta,\theta\nabla u-\theta\nabla\zeta)_{L^2(\Omega_t)}\,dt \geq 0.
  \end{align*}
  Now, let $\zeta=\zeta_\sigma:=u-\sigma\psi$ with $\sigma\in(0,1)$ and $\psi\in\mathcal{D}_{W^{1,p}}$.
  Then,
  \begin{align} \label{Pf_Ch:sigma}
    \sigma\int_0^T(\mathbf{w}-|\nabla\zeta_\sigma|^{p-2}\nabla\zeta_\sigma,\theta\nabla\psi)_{L^2(\Omega_t)}\,dt \geq 0.
  \end{align}
  By \eqref{E:p_Lip}, $|\nabla\zeta_\sigma|\leq|\nabla u|+|\nabla\psi|$, \eqref{Pf_IV:ab}, $\zeta_\sigma-u=-\sigma\psi$, and Young's inequality,
  \begin{align*}
    \bigl|\,|\nabla\zeta_\sigma|^{p-2}\nabla\zeta_\sigma-|\nabla u|^{p-2}\nabla u\,\bigr| &\leq c\sigma(|\nabla u|^{p-2}+|\nabla\psi|^{p-2})|\nabla\psi| \\
    &\leq c\sigma(|\nabla u|^{p-1}+|\nabla \psi|^{p-1}) \quad\text{on}\quad Q_T.
  \end{align*}
  Thus, we see by H\"{o}lder's inequality, $u\in L_{W^{1,p}}^p$, and $\psi\in\mathcal{D}_{W^{1,p}}$ that
  \begin{align*}
    \bigl\|\,|\nabla\zeta_\sigma|^{p-2}\nabla\zeta_\sigma-|\nabla u|^{p-2}\nabla u\,\bigr\|_{L_{L^{p'}}^{p'}} &\leq c\sigma\Bigl(\bigl\|\,|\nabla u|^{p-1}\,\bigr\|_{L_{L^{p'}}^{p'}}+\bigl\|\,|\nabla\psi|^{p-1}\,\bigr\|_{L_{L^{p'}}^{p'}}\Bigr) \\
    &= c\sigma\Bigl(\|\nabla u\|_{L_{L^p}^p}^{p-1}+\|\nabla\psi\|_{L_{L^p}^p}^{p-1}\Bigr) \to 0
  \end{align*}
  as $\sigma\to0$.
  Based on this result, we divide \eqref{Pf_Ch:sigma} by $\sigma>0$ and send $\sigma\to0$ to get
  \begin{align*}
    \int_0^T(\mathbf{w}-|\nabla u|^{p-2}\nabla u,\theta\nabla\psi)_{L^2(\Omega_t)}\,dt \geq 0.
  \end{align*}
  Replacing $\psi$ by $-\psi$, we also have the opposite inequality.
  Therefore,
  \begin{align*}
    \int_0^T(\mathbf{w}-|\nabla u|^{p-2}\nabla u,\theta\nabla\psi)_{L^2(\Omega_t)}\,dt = 0
  \end{align*}
  for all $\psi\in\mathcal{D}_{W^{1,p}}$ and nonnegative $\theta\in\mathcal{D}(0,T)$.
  Moreover, assuming that
  \begin{align*}
    \psi(\cdot,t) = 0 \quad\text{on}\quad \Omega_t \quad\text{for all}\quad t\in[0,\delta]\cup[T-\delta,T]
  \end{align*}
  with some $\delta>0$, and taking a nonnegative $\theta$ such that $\theta=1$ on $[\delta,T-\delta]$, we have
  \begin{align*}
    \int_0^T(\mathbf{w}-|\nabla u|^{p-2}\nabla u,\nabla\psi)_{L^2(\Omega_t)}\,dt = 0 \quad\text{for all}\quad \psi\in\mathcal{D}_{W^{1,p}}.
  \end{align*}
  This equality is also valid for $\psi\in L_{W^{1,p}}^p$ by a density argument, since $\mathcal{D}_{W^{1,p}}$ is dense in $L_{W^{1,p}}^p$.
  Therefore, we obtain \eqref{E:Chara}.
\end{proof}

Propositions \ref{P:Lim_Wpp} and \ref{P:Chara} show that $u\in\mathbb{W}^{p,p'}$ satisfies \eqref{E:pLap_WF} and $u(0)=u_0$ in $L^2(\Omega_0)$.
Thus, $u$ is a weak solution to \eqref{E:pLap_MoDo}, and the uniqueness follows from Proposition \ref{P:Uni}.
This completes the proof of Theorem \ref{T:Exi_Uni}.

\section{Regularity of the time derivative} \label{S:Reg}
In this section, we show that the time derivative of the weak solution $u$ exists in the $L^2$ sense when the given data $f$ and $u_0$ have a better regularity.
Note that, by $p>2$,
\begin{align*}
  W^{1,p}(\Omega_t) \subset L^2(\Omega_t), \quad L_{L^2}^2 \subset [L_{W^{1,p}}^p]^\ast = L_{[W^{1,p}]^\ast}^{p'}.
\end{align*}

\begin{theorem} \label{T:Reg_dt}
  Under Assumption \ref{A:Domain}, suppose further that
  \begin{align*}
    \mathbf{v}_\Omega\in C^1(\overline{Q_T})^n, \quad f\in L_{L^2}^2, \quad u_0\in W^{1,p}(\Omega_0).
  \end{align*}
  Let $u$ be the unique weak solution to \eqref{E:pLap_MoDo} given in Theorem \ref{T:Exi_Uni}.
  Then,
  \begin{align} \label{E:Reg_dt}
    \begin{aligned}
      &\partial^\bullet u, \, \partial_tu, \, \mathrm{div}(|\nabla u|^{p-2}\nabla u) \in L_{L^2}^2 = L^2(Q_T), \\
      &\partial_tu-\mathrm{div}(|\nabla u|^{p-2}\nabla u) = f \quad\text{a.e. in}\quad Q_T.
    \end{aligned}
  \end{align}
\end{theorem}

\begin{remark} \label{R:Vel_Reg}
  Under Assumption \ref{A:Domain} and Lemma \ref{L:Pinv_Reg}, we already have
  \begin{align*}
    \mathbf{v}_\Omega = [\partial_t\Phi_{(\cdot)}]\circ\Phi_{(\cdot)}^{-1} \in C(\overline{Q_T})^n, \quad \nabla\mathbf{v}_\Omega \in C(\overline{Q_T})^{n\times n}.
  \end{align*}
  Thus, to get $\mathbf{v}_\Omega\in C^1(\overline{Q_T})^n$, it is sufficient to assume further that
  \begin{align*}
    \partial_t^2\Phi_{(\cdot)} \in C(\overline{\Omega_0}\times[0,T])^n.
  \end{align*}
  Also, since $\mathbf{v}_\Omega\in C^1(\overline{Q_T})^n$ and $\partial^\bullet=\partial_t+\mathbf{v}_\Omega\cdot\nabla$, we have
  \begin{align} \label{E:Mat_Vel}
    |\mathbf{v}_\Omega| \leq c, \quad |\nabla\mathbf{v}_\Omega| \leq c, \quad |\partial^\bullet\mathbf{v}_\Omega| \leq c \quad\text{on}\quad \overline{Q_T}.
  \end{align}
\end{remark}

We establish Theorem \ref{T:Reg_dt} by the Galerkin method with a higher order energy estimate.
The main tool is the following differentiation formulas.

\begin{lemma} \label{L:dt_Int}
  Let $\psi$ be a function of the form
  \begin{align*}
    \psi(x,t) = \sum_{k=1}^N\theta_k(t)w_k^t(x), \quad (x,t)\in\overline{Q_T}, \quad N\in\mathbb{N}, \quad \theta_k\in C^1([0,T]).
  \end{align*}
  Then,
  \begin{align} \label{E:dt_Int}
    \begin{aligned}
      \frac{1}{p}\frac{d}{dt}\int_{\Omega_t}|\nabla\psi(x,t)|^p\,dx &= \bigl([|\nabla\psi|^{p-2}\nabla\psi](t),[\nabla\partial^\bullet\psi](t)\bigr)_{L^2(\Omega_t)}+R_1^t(\psi), \\
      \frac{d}{dt}\int_{\Omega_t}[\psi\mathbf{v}_\Omega\cdot\nabla\psi](x,t)\,dx &= \bigl(\psi(t),[\mathbf{v}_\Omega\cdot(\nabla\partial^\bullet\psi)](t)\bigr)_{L^2(\Omega_t)}+R_2^t(\psi)
    \end{aligned}
  \end{align}
  for all $t\in(0,T)$, where
  \begin{align*}
    R_1^t(\psi) &:= \int_{\Omega_t}\left\{-|\nabla\psi|^{p-2}\nabla\psi\cdot[(\nabla\mathbf{v}_\Omega)\nabla\psi]+\frac{1}{p}|\nabla\psi|^p\,\mathrm{div}\,\mathbf{v}_\Omega\right\}\,dx, \\
    R_2^t(\psi) &:= \int_{\Omega_t}\{[\partial^\bullet(\psi\mathbf{v}_\Omega)]\cdot\nabla\psi-\psi\mathbf{v}_\Omega\cdot[(\nabla\mathbf{v}_\Omega)\nabla\psi]+(\psi\mathbf{v}_\Omega\cdot\nabla\psi)\,\mathrm{div}\,\mathbf{v}_\Omega\}\,dx.
  \end{align*}
\end{lemma}

\begin{proof}
  By $\theta_k\in C^1([0,T])$, \eqref{E:EBF_Reg}, $\mathbf{v}_\Omega\in C^1(\overline{Q_T})^n$, and $p>2$, we see that
  \begin{align*}
    \psi \in C^1(\overline{Q_T}), \quad \nabla\psi\in C^1(\overline{Q_T})^n \quad\text{and thus}\quad |\nabla\psi|^p, \, \psi\mathbf{v}_\Omega\cdot\nabla\psi \in C^1(\overline{Q_T}).
  \end{align*}
  By this fact and Assumption \ref{A:Domain}, we can use the Reynolds transport theorem (see \cite{Gur81})
  \begin{align*}
    \frac{d}{dt}\int_{\Omega_t}\chi\,dx = \int_{\Omega_t}(\partial^\bullet\chi+\chi\,\mathrm{div}\,\mathbf{v}_\Omega)\,dx, \quad \chi\in C^1(\overline{Q_T})
  \end{align*}
  to the left-hand sides of \eqref{E:dt_Int}.
  After that, we further apply
  \begin{align*}
    \partial^\bullet(\partial_i\psi) &= (\partial_t+\mathbf{v}_\Omega\cdot\nabla)(\partial_i\psi) = \partial_i(\partial_t\psi+\mathbf{v}_\Omega\cdot\nabla\psi)-(\partial_i\mathbf{v}_\Omega)\cdot\nabla\psi \\
    &= \partial_i(\partial^\bullet\psi)-(\partial_i\mathbf{v}_\Omega)\cdot\nabla\psi \quad\text{on}\quad \overline{Q_T}, \quad i=1,\dots,n,
  \end{align*}
  which is allowed by the regularity of $\psi$.
  Then, we obtain \eqref{E:dt_Int}.
\end{proof}

\subsection{Another approximate solution} \label{SS:Re_Ap}
Let $f\in L_{L^2}^2$ and $u_0\in W^{1,p}(\Omega_0)$.
Since $\mathcal{L}(\{w_k^0\}_k)$ is dense in $W^{1,p}(\Omega_0)$ (see Lemma \ref{L:Basis}), we can take functions
\begin{align*}
  v_{0,N} = \sum_{k=1}^N\beta_{N,0}^kw_k^0, \quad N\in\mathbb{N} \quad\text{such that}\quad \lim_{N\to\infty}\|u_0-v_{0,N}\|_{W^{1,p}(\Omega_0)} = 0.
\end{align*}
Note that $v_{0,N}$ is not an orthogonal projection of $u_0$, but we have
\begin{align} \label{E:v0N}
  \|v_{0,N}\|_{L^2(\Omega_0)} \leq c\|v_{0,N}\|_{L^p(\Omega_0)} \leq c\|v_{0,N}\|_{W^{1,p}(\Omega_0)} \leq c
\end{align}
by H\"{o}lder's inequality and the strong convergence of $v_{0,N}$ in $W^{1,p}(\Omega_0)$.
Also, since $\mathcal{D}_{L^2}$ is dense in $L_{L^2}^2$, we can take functions $g_N\in\mathcal{D}_{L^2}$ such that
\begin{align} \label{E:gN}
  \lim_{N\to\infty}\|f-g_N\|_{L_{L^2}^2} = 0 \quad\text{and thus}\quad \|g_N\|_{L_{L^2}^2} \leq c \quad\text{for all}\quad N\in\mathbb{N}.
\end{align}
For each $N\in\mathbb{N}$, we seek for a function $v_N(t)=\sum_{k=1}^N\beta_N^k(t)w_k^t$ such that
\begin{multline} \label{E:vN_eq}
  \bigl(\partial^\bullet v_N(t),w_k^t\bigr)_{L^2(\Omega_t)}+\bigl([|\nabla v_N|^{p-2}\nabla v_N](t),\nabla w_k^t\bigr)_{L^2(\Omega_t)} \\
  +\bigl(v_N(t),[\mathbf{v}_\Omega\cdot\nabla w_k^t+w_k^t\,\mathrm{div}\,\mathbf{v}_\Omega](t)\bigr)_{L^2(\Omega_t)} = \bigl(g_N(t),w_k^t\bigr)_{L^2(\Omega_t)}
\end{multline}
for all $t\in(0,T)$ and $k=1,\dots,N$, and $v_N(0)=v_{0,N}$.
Here, the function
\begin{align*}
  \bigl(g_N(t),w_k^t\bigr)_{L^2(\Omega_t)} = \int_{\Omega_t}g_N(x,t)w_k^t(x)\,dx = \int_{\Omega_0}[\phi_{-t}g_N(t)](X)w_k^0(X)J_t(X)\,dX
\end{align*}
is continuous on $[0,T]$ by Assumption \ref{A:Domain} and $g_N\in\mathcal{D}_{L^2}$.
Using this fact, \eqref{E:v0N}, and \eqref{E:gN}, we can show as in Sections \ref{SS:Ex_Ap} and \ref{SS:Ex_Ene} that there exists a unique $C^1$ solution to \eqref{E:vN_eq} defined on the whole time interval $[0,T]$ and satisfying
\begin{align} \label{E:vN_Ener}
  \|v_N(t)\|_{L^2(\Omega_t)}^2+\int_0^t\|\nabla v_N(s)\|_{L^p(\Omega_s)}^p\,ds \leq c_T \quad\text{for all}\quad t\in[0,T].
\end{align}
Moreover, as in Sections \ref{SS:Ex_WF}--\ref{SS:Ex_Char}, we see that
\begin{align} \label{E:vN_WeCo}
  \lim_{N\to\infty}v_N = v \quad\text{weakly-$\ast$ in $L_{L^2}^\infty$ and weakly in $L_{W^{1,p}}^p$},
\end{align}
and $v$ is a weak solution to \eqref{E:pLap_MoDo} for the data $f$ and $u_0$.
Hence, $v=u$ by Proposition \ref{P:Uni}, where $u$ is the weak solution to \eqref{E:pLap_MoDo} given in Theorem \ref{T:Exi_Uni}.

\subsection{Higher order energy estimate} \label{SS:Re_HiEn}
Let us derive a higher order energy estimate for $v_N$, which gives a regularity of the weak material derivative of $v=u$.

\begin{proposition} \label{P:HiEn}
  We have
  \begin{align} \label{E:HiEn}
    \frac{1}{2}\int_0^t\|\partial^\bullet v_N(s)\|_{L^2(\Omega_s)}^2\,ds+\frac{1}{2p}\|\nabla v_N(t)\|_{L^p(\Omega_t)}^p \leq c_T \quad\text{for all}\quad t\in[0,T].
  \end{align}
\end{proposition}

\begin{proof}
  Noting that
  \begin{align*}
    \partial^\bullet v_N(t) = \sum_{k=1}^N[\beta_N^k]'(t)w_k^t, \quad \nabla\partial^\bullet v_N(t) = \sum_{k=1}^N[\beta_N^k]'(t)\nabla w_k^t, \quad [\beta_N^k]'(t) = \frac{d\beta_N^k}{dt}(t)
  \end{align*}
  by \eqref{E:MT_EBF}, we multiply \eqref{E:vN_eq} by $[\beta_N^k]'(t)$ and sum over $k=1,\dots,N$.
  Then,
  \begin{multline*}
    \|\partial^\bullet v_N(t)\|_{L^2(\Omega_t)}^2+\bigl([|\nabla v_N|^{p-2}\nabla v_N](t),[\nabla\partial^\bullet v_N](t)\bigr)_{L^2(\Omega_t)} \\
    +\bigl(v_N(t),[\mathbf{v}_\Omega\cdot(\nabla\partial^\bullet v_N)+(\partial^\bullet v_N)\,\mathrm{div}\,\mathbf{v}_\Omega](t)\bigr)_{L^2(\Omega_t)} = \bigl(g_N(t),\partial^\bullet v_N(t)\bigr)_{L^2(\Omega_t)}.
  \end{multline*}
  Moreover, we apply \eqref{E:dt_Int} with $\psi=v_N$ to find that
  \begin{multline*}
    \|\partial^\bullet v_N(t)\|_{L^2(\Omega_t)}^2+\frac{1}{p}\frac{d}{dt}\|\nabla v_N(t)\|_{L^p(\Omega_t)}^p+\frac{d}{dt}\bigl(v_N(t),[\mathbf{v}_\Omega\cdot\nabla v_N](t)\bigr)_{L^2(\Omega_t)} \\
    = \bigl(g_N(t),\partial^\bullet v_N(t)\bigr)_{L^2(\Omega_t)}-\bigl(v_N(t),[(\partial^\bullet v_N)\,\mathrm{div}\,\mathbf{v}_\Omega](t)\bigr)_{L^2(\Omega_t)}+R_1^t(v_N)+R_2^t(v_N).
  \end{multline*}
  We integrate both sides over $(0,t)$ and move the term
  \begin{multline*}
    \int_0^t\frac{d}{ds}\bigl(v_N(s),[\mathbf{v}_\Omega\cdot\nabla v_N](s)\bigr)_{L^2(\Omega_s)}\,ds \\
    = \bigl(v_N(t),[\mathbf{v}_\Omega\cdot\nabla v_N](t)\bigr)_{L^2(\Omega_t)}-(v_{0,N},\mathbf{v}_\Omega(0)\cdot\nabla v_{0,N})_{L^2(\Omega_0)}
  \end{multline*}
  to the right-hand side.
  Then, we further apply \eqref{E:Mat_Vel} and H\"{o}lder's and Young's inequalities, and then use \eqref{E:Lp_L2} with a suitable $\delta>0$.
  As a result, we obtain
  \begin{multline} \label{Pf_HiEn:Ineq}
    \int_0^t\|\partial^\bullet v_N(s)\|_{L^2(\Omega_s)}^2\,ds+\frac{1}{p}\|\nabla v_N(t)\|_{L^p(\Omega_t)}^p-\frac{1}{p}\|\nabla v_{0,N}\|_{L^p(\Omega_0)}^p \\
    \leq \frac{1}{2}\int_0^t\|\partial^\bullet v_N(s)\|_{L^2(\Omega_s)}^2\,ds+\frac{1}{2p}\|\nabla v_N(t)\|_{L^p(\Omega_t)}^p+cI_N(t)
  \end{multline}
  for all $t\in[0,T]$, where
  \begin{multline*}
    I_N(t) := 1+\|v_{0,N}\|_{L^2(\Omega_0)}^2+\|\nabla v_{0,N}\|_{L^p(\Omega_0)}^p+\|v_N(t)\|_{L^2(\Omega_t)}^2 \\
    +\int_0^t\Bigl(1+\|v_N(s)\|_{L^2(\Omega_s)}^2+\|\nabla v_N(s)\|_{L^p(\Omega_s)}^p+\|g_N(t)\|_{L^2(\Omega_s)}^2\Bigr)\,ds.
  \end{multline*}
  Moreover, we observe by \eqref{E:v0N}, \eqref{E:gN}, and \eqref{E:vN_Ener} that
  \begin{align*}
    \frac{1}{p}\|\nabla v_{0,N}\|_{L^p(\Omega_0)}^p \leq c, \quad I_N(t) \leq c_T \quad\text{for all}\quad t\in[0,T].
  \end{align*}
  Thus, we get \eqref{E:HiEn} by making the first two terms on the right-hand side of \eqref{Pf_HiEn:Ineq} absorbed into the left-hand side and by using the above estimates.
\end{proof}

It follows from \eqref{E:vN_WeCo} with $v=u$ and \eqref{E:HiEn} that
\begin{align*}
  \lim_{N\to\infty}\partial^\bullet v_N = \partial^\bullet u \quad\text{weakly in $L_{L^2}^2$}, \quad \lim_{N\to\infty}\nabla v_N = \nabla u \quad\text{weakly-$\ast$ in $L_{L^p}^\infty$}.
\end{align*}
By this result, $u\in L_{W^{1,p}}^p$, and $\mathbf{v}_\Omega\in C^1(\overline{Q_T})^n$, we find that
\begin{align*}
  \partial^\bullet u\in L_{L^2}^2, \quad \partial_tu = \partial^\bullet u-\mathbf{v}_\Omega\cdot\nabla u \in L_{L^2}^2.
\end{align*}
Now, for $\psi\in L_{W^{1,p}}^p$, we can write \eqref{E:pLap_WF} as
\begin{align*}
  (|\nabla u|^{p-2}\nabla u,\nabla\psi)_{L^2(Q_T)} = (f-\partial^\bullet u-u\,\mathrm{div}\,\mathbf{v}_\Omega,\psi)_{L^2(Q_T)}-(u\mathbf{v}_\Omega,\nabla\psi)_{L^2(Q_T)}.
\end{align*}
Moreover, when $\psi\in\mathcal{D}(Q_T)$, we carry out integration by parts to get
\begin{align*}
  (|\nabla u|^{p-2}\nabla u,\nabla\psi)_{L^2(Q_T)} &= (f-\partial^\bullet u-u\,\mathrm{div}\,\mathbf{v}_\Omega,\psi)_{L^2(Q_T)}+(\mathrm{div}(u\mathbf{v}_\Omega),\psi)_{L^2(Q_T)} \\
  &= (f-\partial_tu,\psi)_{L^2(Q_T)}.
\end{align*}
Hence, $-\mathrm{div}(|\nabla u|^{p-2}\nabla u)=f-\partial_tu \in L^2(Q_T)$ and Theorem \ref{T:Reg_dt} is valid.

\subsection{Boundary condition} \label{SS:Re_Bo}
We also find that the boundary condition of \eqref{E:pLap_MoDo} is satisfied in the following weak sense.
Noting that $1/p+1/p'=1$, we write
\begin{align*}
  W^{1-1/p,p}(\partial\Omega_t) := \{v|_{\partial\Omega_t} \mid v\in W^{1,p}(\Omega_t)\}, \quad W^{-1/p',p'}(\partial\Omega_t) := [W^{1-1/p,p}(\partial\Omega_t)]^\ast.
\end{align*}

\begin{proposition} \label{P:BC_Weak}
  Under the assumptions of Theorem \ref{T:Reg_dt}, we have
  \begin{align} \label{E:BC_Weak}
    [|\nabla u|^{p-2}\partial_\nu u+V_\Omega u](t) = 0 \quad\text{in}\quad W^{-1/p',p'}(\partial\Omega_t)
  \end{align}
  for almost all $t\in(0,T)$.
\end{proposition}

\begin{proof}
  For a.a. $t\in(0,T)$, we see by $u\in L_{W^{1,p}}^p$, $\mathbf{v}_\Omega\in C^1(\overline{Q_T})^n$, and \eqref{E:Reg_dt} that
  \begin{align*}
    \mathbf{z}(t) := [|\nabla u|^{p-2}\nabla u+u\mathbf{v}_\Omega](t) \in L^{p'}(\Omega_t)^n, \quad \mathrm{div}\,\mathbf{z}(t) \in L^2(\Omega_t) \subset L^{p'}(\Omega_t).
  \end{align*}
  Thus, by \cite[Theorem III.2.2]{Gal11}, we can consider the normal trace
  \begin{align*}
    \mathbf{z}(t)\cdot\bm{\nu}(\cdot,t) = [|\nabla u|^{p-2}\partial_\nu u+V_\Omega u](t) \quad\text{in}\quad W^{-1/p',p'}(\partial\Omega_t)
  \end{align*}
  and the following integration by parts formula holds for $\psi\in W^{1,p}(\Omega_t)$:
  \begin{align} \label{Pf_BCW:GIbP}
    (\mathbf{z}(t),\nabla\psi)_{L^2(\Omega_t)} = \langle[|\nabla u|^{p-2}\partial_\nu u+V_\Omega u](t),\psi\rangle_{W^{1-1/p,p}(\partial\Omega_t)}-(\mathrm{div}\,\mathbf{z}(t),\psi)_{L^2(\Omega_t)}.
  \end{align}
  Let $\theta\in\mathcal{D}(0,T)$ and $\Psi_0\in C^1(\overline{\Omega_0})$.
  Since $\theta(\cdot)[\phi_{(\cdot)}\Psi_0]\in L_{W^{1,p}}^p$ by \eqref{E:LBq_Equi}, we have
  \begin{align*}
    \int_0^T\theta(t)\bigr(\mathbf{z}(t),\nabla(\phi_t\Psi_0)\bigr)_{L^2(\Omega_t)}\,dt = \int_0^T\theta(t)([f-\partial^\bullet u-u\,\mathrm{div}\,\mathbf{v}_\Omega](t),\phi_t\Psi_0)_{L^2(\Omega_t)}\,dt
  \end{align*}
  by setting $\psi=\theta(\cdot)[\phi_{(\cdot)}\Psi_0]$ in \eqref{E:pLap_WF}.
  This holds for any $\theta$, and thus
  \begin{align*}
    \bigl(\mathbf{z}(t),\nabla(\phi_t\Psi_0)\bigr)_{L^2(\Omega_t)} = ([f-\partial^\bullet u-u\,\mathrm{div}\,\mathbf{v}_\Omega](t),\phi_t\Psi_0)_{L^2(\Omega_t)}
  \end{align*}
  for a.a. $t\in(0,T)$.
  We apply \eqref{E:Reg_dt}, \eqref{Pf_BCW:GIbP}, and $\partial^\bullet u=\partial_tu+\mathbf{v}_\Omega\cdot\nabla u$ to get
  \begin{align*}
    \langle[|\nabla u|^{p-2}\partial_\nu u+V_\Omega u](t),\phi_t\Psi_0\rangle_{W^{1-1/p,p}(\partial\Omega_t)} = 0 \quad\text{for all}\quad \Psi_0\in C^1(\overline{\Omega_0}).
  \end{align*}
  Since $\phi_t\colon C^1(\overline{\Omega_0})\to C^1(\overline{\Omega_t})$ is a bijection by Assumption \ref{A:Domain}, it follows that
  \begin{align*}
    \langle[|\nabla u|^{p-2}\partial_\nu u+V_\Omega u](t),\Psi_t\rangle_{W^{1-1/p,p}(\partial\Omega_t)} = 0 \quad\text{for all}\quad \Psi_t\in C^1(\overline{\Omega_t}).
  \end{align*}
  Hence, by the density of $C^1(\overline{\Omega_t})$ in $W^{1,p}(\Omega_t)$, we conclude that \eqref{E:BC_Weak} is valid.
\end{proof}

\begin{remark} \label{R:BC_Weak}
  If we could get the second order regularity
  \begin{align*}
    |\nabla u|^{p-2}\partial_iu \in L_{H^1}^2, \quad i=1,\dots,n,
  \end{align*}
  then the boundary condition of \eqref{E:pLap_MoDo} would hold a.e. on $\partial_\ell Q_T$.
  Such a regularity result was recently shown in the case of a non-moving domain and the Dirichlet boundary condition \cite{CiaMaz20}, which relies on a regularity result for the elliptic $p$-Laplace equation \cite{CiaMaz18,CiaMaz19}.
  It seems that the arguments of \cite{CiaMaz18,CiaMaz19} do not directly apply the case of the boundary condition of \eqref{E:pLap_MoDo} and some further discussions will be required.
  This is out of the scope of this paper, so we leave the above regularity problem open.
\end{remark}

\section{Extension to a Leray--Lions type operator} \label{S:LerLio}
In this section, we briefly discuss the existence and uniqueness of a weak solution to the parabolic equation with a Leray--Lions type operator
\begin{align} \label{E:Para_LL}
  \left\{
  \begin{alignedat}{3}
    \partial_tu-\mathrm{div}\bigl(\mathbf{A}(x,t,u,\nabla u)\bigr) &= f &\quad &\text{in} &\quad &Q_T , \\
    \mathbf{A}(x,t,u,\nabla u)\cdot\bm{\nu}+V_\Omega u &= 0 &\quad &\text{on} &\quad &\partial_\ell Q_T, \\
    u|_{t=0} &= u_0 &\quad &\text{in} &\quad &\Omega_0,
  \end{alignedat}
  \right.
\end{align}
Here, $\mathbf{A}=\mathbf{A}(x,t,u,\mathbf{z})$ is a vector-valued function
\begin{align*}
  \mathbf{A}\colon(Q_T\cup\partial_\ell Q_T)\times\mathbb{R}\times\mathbb{R}^n \to \mathbb{R}^n.
\end{align*}
To extend the result of Theorem \ref{T:Exi_Uni} to \eqref{E:Para_LL}, we make the following assumptions.

\begin{assumption} \label{A:LLOP}
  The function $\mathbf{A}$ is a Carath\'{e}odory function, i.e.,
  \begin{itemize}
    \item $Q_T\ni(x,t)\mapsto\mathbf{A}(x,t,u,\mathbf{z})$ is measurable for all fixed $(u,\mathbf{z})\in\mathbb{R}\times\mathbb{R}^n$,
    \item $\mathbb{R}\times\mathbb{R}^n\ni(u,\mathbf{z})\mapsto\mathbf{A}(x,t,u,\mathbf{z})$ is continuous for a.a. fixed $(x,t)\in Q_T$.
  \end{itemize}
  Moreover, there exist constants $c,p,p'>0$ and functions $\Lambda_1,\Lambda_2$ such that
  \begin{align*}
    p \in (2,\infty), \quad \frac{1}{p}+\frac{1}{p'} = 1, \quad \Lambda_1\in L_{L^{p'}}^{p'}, \quad \Lambda_2\in L_{L^1}^1
  \end{align*}
  and the inequalities
  \begin{gather}
    |\mathbf{A}(x,t,u,\mathbf{z})| \leq c\{|u|^{p-1}+|\mathbf{z}|^{p-1}+\Lambda_1(x,t)\}, \label{E:LL_Bdd} \\
    \mathbf{A}(x,t,u,\mathbf{z})\cdot\mathbf{z} \geq c|\mathbf{z}|^p-\Lambda_2(x,t), \label{E:LL_Coer} \\
    \{\mathbf{A}(x,t,u,\mathbf{z}_1)-\mathbf{A}(x,t,u,\mathbf{z}_2)\}\cdot(\mathbf{z}_1-\mathbf{z}_2) \geq 0, \label{E:LL_Mono} \\
    |\mathbf{A}(x,t,u_1,\mathbf{z})-\mathbf{A}(x,t,u_2,\mathbf{z})| \leq c|u_1-u_2||\mathbf{z}|^{p-1} \label{E:LL_Con}
  \end{gather}
  hold for all $(x,t)\in Q_T$ and $u,u_1,u_2\in\mathbb{R}$ and $\mathbf{z},\mathbf{z}_1,\mathbf{z}_2\in\mathbb{R}^n$.
\end{assumption}

Here, the conditions \eqref{E:LL_Bdd}--\eqref{E:LL_Mono} are standard in the literature (see e.g. \cite{LaSoUr68,Lio69}).
Also, the condition \eqref{E:LL_Con} is a simplified version of \cite[(2.6)]{CaNoOr17}.
Note that the parabolic $p$-Laplace equation \eqref{E:pLap_MoDo} corresponds to the case $\mathbf{A}(x,t,u,\mathbf{z})=|\mathbf{z}|^{p-2}\mathbf{z}$.

We define a weak solution to \eqref{E:Para_LL} as in Definition \ref{D:WS_pLap}, where we use the exponent $p$ in the above assumption and replace $|\nabla u|^{p-2}\nabla u$ by $\mathbf{A}(\cdot,\cdot,u,\nabla u)$.
Then, we can show that the existence and uniqueness of a weak solution to \eqref{E:Para_LL} hold as in Theorem \ref{T:Exi_Uni}.

\begin{theorem} \label{T:LL_ExUn}
  For all $f\in L_{[W^{1,p}]^\ast}^{p'}$ and $u_0\in L^2(\Omega_0)$, there exists a unique weak solution to \eqref{E:Para_LL}.
\end{theorem}

Since the proof of Theorem \ref{T:LL_ExUn} is similar to that of Theorem \ref{T:Exi_Uni}, we do not repeat the whole proof.
Instead, we explain where we modify the proof below.

\subsection{Uniqueness} \label{SS:LL_Un}
In the proof of Proposition \ref{P:Uni}, we replace
\begin{align*}
  \mathbf{b} = |\nabla u_1|^{p-2}\nabla u_1-|\nabla u_2|^{p-2}\nabla u_2, \quad \mathcal{I}_\varepsilon^2(t) := \int_0^t\Bigl(\mathbf{b}(s),\bigl[\nabla[\zeta_\varepsilon'(v)]\bigr](s)\Bigr)_{L^2(\Omega_s)}\,ds
\end{align*}
by $\mathbf{d}:=\mathbf{d}_1+\mathbf{d}_2$ and $\mathcal{K}_\varepsilon(t):=\mathcal{K}_\varepsilon^1(t)+\mathcal{K}_\varepsilon^2(t)$, where
\begin{align*}
  \mathbf{d}_1 &:= \mathbf{A}(x,t,u_1,\nabla u_1)-\mathbf{A}(x,t,u_1,\nabla u_2), \\
  \mathbf{d}_2 &:= \mathbf{A}(x,t,u_1,\nabla u_2)-\mathbf{A}(x,t,u_2,\nabla u_2), \\
  \mathcal{K}_\varepsilon^j(t) &:= \int_0^t\Bigl(\mathbf{d}_j(s),\bigl[\nabla[\zeta_\varepsilon'(v)]\bigr](s)\Bigr)_{L^2(\Omega_s)}\,ds, \quad j=1,2.
\end{align*}
Noting that $v=u_1-u_2$ and $\nabla[\zeta_\varepsilon'(v)]=\zeta_\varepsilon''(v)\nabla v$, we have $\mathcal{K}_\varepsilon^1(t)\geq0$ by \eqref{Pf_Un:abs} and \eqref{E:LL_Mono} as in \eqref{Pf_Un:I2}.
Next, we apply \eqref{E:LL_Con} to $\mathbf{d}_2$ and find that
\begin{align*}
  |\mathcal{K}_\varepsilon^2(t)| \leq c\int_{Q_T}\bigl[\zeta_\varepsilon''(v)|v||\nabla u_2|^{p-1}|\nabla v|\bigr](x,s)\,dxds.
\end{align*}
Moreover, $|\nabla u_2|^{p-1}|\nabla v|\in L_{L^1}^1$ by $u_2,v\in L_{W^{1,p}}^p$ and H\"{o}lder's inequality.
Thus, as in \eqref{Pf_Un:I3}, we can get $\mathcal{K}_\varepsilon^2(t)\to0$ as $\varepsilon\to0$.
The other parts of the proof remain the same.

\subsection{Approximate solutions} \label{SS:LL_Ap}
Instead of \eqref{E:Approx}, we consider
\begin{multline} \label{E:LL_App}
  \bigl(\partial^\bullet u_N(t),w_k^t\bigr)_{L^2(\Omega_t)}+\bigl([\mathbf{A}(\cdot,\cdot,u_N,\nabla u_N)](t),\nabla w_k^t\bigr)_{L^2(\Omega_t)} \\
  +\bigl(u_N(t),[\mathbf{v}_\Omega\cdot\nabla w_k^t+w_k^t\,\mathrm{div}\,\mathbf{v}_\Omega](t)\bigr)_{L^2(\Omega_t)} = \langle f_N(t),w_k^t\rangle_{W^{1,p}(\Omega_t)}
\end{multline}
and $u_N(0)=u_{0,N}$ for $u_N(t)=\sum_{k=1}^N\alpha_N^k(t)w_k^t$.
This is equivalent to
\begin{align} \label{E:LL_ODE}
  \left\{
  \begin{alignedat}{2}
    \frac{d\bm{\alpha}_N}{dt}(t) &= \mathbf{h}_N(\bm{\alpha}_N(t),t), &\quad &\mathbf{h}_N(\mathbf{a},t) := \mathbf{M}_N(t)^{-1}[\mathbf{f}_N(t)-\bm{\zeta}_N(\mathbf{a},t)-\mathbf{B}_N(t)\mathbf{a}], \\
    \alpha_N^k(0) &= (u_0,w_k^0)_{L^2(\Omega_0)}, &\quad &k=1,\dots,N.
  \end{alignedat}
  \right.
\end{align}
Here, we use the notations in Section \ref{SS:Ex_Ap}.
Also, we set
\begin{align*}
  \zeta_N^k(\mathbf{a},t) &:= \bigl([\mathbf{A}(\cdot,\cdot,u_{\mathbf{a}},\nabla u_{\mathbf{a}})](t),\nabla w_k^t\bigr)_{L^2(\Omega_t)}, \quad u_{\mathbf{a}}(t) := \sum_{k=1}^Na_kw_k^t
\end{align*}
for $\mathbf{a}=(a_1,\dots,a_N)^{\mathrm{T}}\in\mathbb{R}^N$ and write $\bm{\zeta}_N(\mathbf{a},t)=\bigl(\zeta_N^k(\mathbf{a},t)\bigr)_k$.

In general, $\mathbf{h}_N$ is not continuous in $t$ and locally Lipschitz continuous in $\mathbf{a}$.
Thus, we use the Carath\'{e}odory existence theorem to get a solution to \eqref{E:LL_ODE}.
We can see that
\begin{itemize}
  \item $[0,T]\ni t\mapsto\mathbf{h}_N(\mathbf{a},t)$ is measurable for all fixed $\mathbf{a}\in\mathbb{R}^N$,
  \item $\mathbb{R}^N\ni\mathbf{a}\mapsto\mathbf{h}_N(\mathbf{a},t)$ is continuous for a.a. fixed $t\in[0,T]$
\end{itemize}
by Assumptions \ref{A:Domain} and \ref{A:LLOP}, Lemma \ref{L:Uni_PD}, the definitions \eqref{E:Def_fNk} and \eqref{E:Def_BN} of $\mathbf{f}_N$ and $\mathbf{B}_N$, the conditions \eqref{E:EBF_Reg} and \eqref{E:FN_DWp}, the change of variables $x=\Phi_t(X)$ (for the measurability of $\zeta_N^k$ in $t$), and the dominated convergence theorem (for the continuity of $\zeta_N^k$ in $\mathbf{a}$).
Also,
\begin{align*}
  |\zeta_N^k(\mathbf{a},t)| &\leq c\Bigl(\|u_{\mathbf{a}}(t)\|_{L^p(\Omega_t)}^{p-1}+\|\nabla u_{\mathbf{a}}(t)\|_{L^p(\Omega_t)}^{p-1}+\|\Lambda_1(t)\|_{L^{p'}(\Omega_t)}\Bigr)\|\nabla w_k^t\|_{L^p(\Omega_t)} \\
  &\leq c_T\Bigl(|\mathbf{a}|^{p-1}+\|\Lambda_1(t)\|_{L^{p'}(\Omega_t)}\Bigr)
\end{align*}
by \eqref{E:LL_Bdd}, H\"{o}lder's inequality, $u_{\mathbf{a}}(t)=\sum_{k=1}^Na_kw_k^t$, and \eqref{E:EBF_Reg}.
Hence, we get
\begin{gather*}
  |\mathbf{h}_N(\mathbf{a},t)| \leq c_T\Bigl(1+|\mathbf{a}|+|\mathbf{a}|^{p-1}+\|\Lambda_1(t)\|_{L^{p'}(\Omega_t)}\Bigr), \quad (\mathbf{a},t)\in\mathbb{R}^N\times[0,T], \\
  \|\Lambda_1(\cdot)\|_{L^{p'}(\Omega_{(\cdot)})} \in L^{p'}(0,T) \subset L^1(0,T).
\end{gather*}
Thus, if we have an a priori bound of a solution to \eqref{E:LL_ODE}, then we can get the existence of a solution to \eqref{E:LL_ODE} on the whole time interval $[0,T]$ by using the Carath\'{e}odory existence theorem repeatedly (note that the uniqueness of a solution to \eqref{E:LL_ODE} is not required).

Let us derive an a priori bound of a solution to \eqref{E:LL_ODE}.
Suppose that a solution $\bm{\alpha}_N$ exists on some time interval $[0,\tau]$ with $\tau\leq T$.
For the corresponding $u_N$, we proceed as in the proofs of Propositions \ref{P:Est_Ave} and \ref{P:Energy} by using \eqref{E:LL_Coer}.
Then, we have \eqref{E:Est_W1p} and
\begin{align*}
  &\|u_N(t)\|_{L^2(\Omega_t)}^2+\int_0^t\|\nabla u_N(s)\|_{L^p(\Omega_s)}^p\,ds \\
  &\qquad \leq c_T\left\{\|u_0\|_{L^2(\Omega_0)}^2+\int_0^t\Bigl(\|u_N(s)\|_{L^2(\Omega_s)}^2+\xi_N(s)\Bigr)\,ds\right\}
\end{align*}
for all $t\in[0,\tau]$, where $\xi_N(s):=\|\Lambda_2(s)\|_{L^1(\Omega_s)}+\|f_N(s)\|_{[W^{1,p}(\Omega_s)]^\ast}^{p'}+1$.
Since
\begin{align*}
  \int_0^\tau\xi_N(s)\,ds \leq \int_0^T\xi_N(s)\,ds \leq \|\Lambda_2\|_{L_{L^1}^1}+\|f_N\|_{L_{[W^{1,p}]^\ast}^{p'}}^{p'}+T \leq c_T
\end{align*}
by $\Lambda_2\in L_{L^1}^1$ and \eqref{E:fN_StCo}, we observe by Gronwall's inequality that the energy estimate \eqref{E:Energy} holds for all $t\in[0,\tau]$.
Thus, by \eqref{E:Uni_PD} and \eqref{E:Energy}, we obtain the a priori bound
\begin{align*}
  |\bm{\alpha}_N(t)|^2 \leq c\bm{\alpha}_N(t)\cdot[\mathbf{M}_N(t)\bm{\alpha}_N(t)] = c\|u_N(t)\|_{L^2(\Omega_t)}^2 \leq c_T
\end{align*}
for all $t\in[0,\tau]$, where $c_T>0$ depends on $T$ but is independent of $\tau$.

\subsection{Convergence and characterization} \label{SS:LL_Con}
For the approximate solution $u_N$, we observe by \eqref{E:Est_W1p}, \eqref{E:Energy}, and \eqref{E:LL_Bdd} that the functions
\begin{align*}
  u_N \in L_{L^2}^\infty\cap L_{W^{1,p}}^p, \quad \mathbf{A}(\cdot,\cdot,u_N,\nabla u_N) \in [L_{L^{p'}}^{p'}]^n
\end{align*}
are bounded uniformly in $N\in\mathbb{N}$.
Thus, as in \eqref{E:WeConv}, we have
\begin{alignat*}{2}
  \lim_{N\to\infty}u_N &= u &\quad &\text{weakly-$\ast$ in $L_{L^2}^\infty$ and weakly in $L_{W^{1,p}}^p$}, \\
  \lim_{N\to\infty}\mathbf{A}(\cdot,\cdot,u_N,\nabla u_N) &= \mathbf{w} &\quad &\text{weakly in $[L_{L^{p'}}^{p'}]^n$}.
\end{alignat*}
Replacing $|\nabla\varphi|^{p-2}\nabla\varphi$ with $\mathbf{A}(\cdot,\cdot,\varphi,\nabla\varphi)$ and using Assumption \ref{A:LLOP}, we can proceed completely in the same way as Sections \ref{SS:Ex_WF}--\ref{SS:Ex_Char} up to \eqref{Pf_Ch:sigma} (divided by $\sigma>0$):
\begin{align*}
  \int_0^T(\mathbf{w}-\mathbf{A}(\cdot,\cdot,\zeta_\sigma,\nabla\zeta_\sigma),\theta\nabla\psi)_{L^2(\Omega_t)}\,dt \geq 0,
\end{align*}
where $\theta\in\mathcal{D}(0,T)$, $\theta\geq0$ and $\zeta_\sigma=u-\sigma\psi$ with $\sigma\in(0,1)$, $\psi\in\mathcal{D}_{W^{1,p}}$.
Now,
\begin{align*}
  \lim_{\sigma\to0}\mathbf{A}(\cdot,\cdot,\zeta_\sigma,\nabla\zeta_\sigma)\cdot(\theta\nabla\psi) = \mathbf{A}(\cdot,\cdot,u,\nabla u)\cdot(\theta\nabla\psi) \quad\text{a.e. in}\quad Q_T,
\end{align*}
since $\mathbf{A}$ is a Carath\'{e}odory function.
Moreover, ,
\begin{align*}
  &|\mathbf{A}(\cdot,\cdot,\zeta_\sigma,\nabla\zeta_\sigma)\cdot(\theta\nabla\psi)| \\
  &\quad \leq c(|\zeta_\sigma|^{p-1}+|\nabla\zeta_\sigma|^{p-1}+\Lambda_1)|\nabla\psi| \\
  &\quad \leq c(|u|^{p-1}+|\psi|^{p-1}+|\nabla u|^{p-1}+|\nabla\psi|^{p-1}+\Lambda_1)|\nabla\psi|
\end{align*}
on $Q_T$ by \eqref{E:LL_Bdd}, \eqref{Pf_IV:ab}, and $\sigma\in(0,1)$.
Since the last line is in $L_{L^1}^1$ by H\"{o}lder's inequality, we can apply the dominated convergence theorem to find that
\begin{align*}
  \lim_{\sigma\to0}\int_0^T(\mathbf{A}(\cdot,\cdot,\zeta_\sigma,\nabla\zeta_\sigma),\theta\nabla\psi)_{L^2(\Omega_t)}\,dt = \int_0^T(\mathbf{A}(\cdot,\cdot,u,\nabla u),\theta\nabla\psi)_{L^2(\Omega_t)}\,dt,
\end{align*}
and the remaining part of the proof of Proposition \ref{P:Chara} is the same.
This completes the modifications required in the proof of Theorem \ref{T:LL_ExUn}.

\section{Concluding remarks} \label{S:Concl}
We should note that the arguments of this paper heavily rely on the setting $p>2$.
The case $1<p<2$ seems to be more difficult and require further discussions.

Our approach is also applicable to the parabolic $p$-Laplace equation (and a Leray--Lions type equation) on a closed hypersurface $\Gamma_t$ in $\mathbb{R}^n$ moving with velocity $\mathbf{v}_\Gamma$:
\begin{align} \label{E:pLap_Sur}
  \partial_\Gamma^\bullet\eta+\eta\,\mathrm{div}_\Gamma\mathbf{v}_\Gamma-\mathrm{div}_\Gamma(|\nabla_\Gamma\eta|^{p-2}\nabla_\Gamma\eta) = f \quad\text{on}\quad \Gamma_t, \quad t\in(0,T).
\end{align}
Here, $\partial_\Gamma^\bullet$, $\nabla_\Gamma$, and $\mathrm{div}_\Gamma$ are the material derivative along $\mathbf{v}_\Gamma$, the tangential gradient, and the surface divergence, respectively.
When $f=0$, the equation \eqref{E:pLap_Sur} comes from the local conservation law
\begin{align*}
  \frac{d}{dt}\int_{\mathcal{M}_t}\eta\,d\mathcal{H}^{n-1} = -\int_{\partial\mathcal{M}_t}\mathbf{q}\cdot\bm{\mu}\,d\mathcal{H}^{n-2} \quad\text{for any}\quad \mathcal{M}_t \subset \Gamma_t
\end{align*}
with flux $\mathbf{q}=-|\nabla_\Gamma\eta|^{p-2}\nabla_\Gamma\eta$, where $\mathcal{H}^k$ is the Hausdorff measure of dimension $k$ and $\bm{\mu}$ is the co-normal on $\partial\mathcal{M}_t$ (see e.g. \cite{DziEll07,DziEll13_AN}).
Note that, contrary to the case of \eqref{E:pLap_MoDo}, a weak form of \eqref{E:pLap_Sur} does not involve the most problematic term
\begin{align*}
  \int_0^T\bigl(\eta(t),[\mathbf{v}_\Gamma\cdot\nabla_\Gamma\psi](t)\bigr)_{L^2(\Gamma_t)}\,dt.
\end{align*}
Thus, the proof of the existence of a weak solution will be simplified for \eqref{E:pLap_Sur}.
Also, due to the absence of the above term, the arguments of \cite[Section 7.1]{AlCaDjEl23} are applicable to \eqref{E:pLap_Sur}, which only use weak convergence results.
In \cite[Example 7.3]{AlCaDjEl23}, the modified equation
\begin{align*}
  \partial_\Gamma^\bullet\eta+\eta\,\mathrm{div}_\Gamma\mathbf{v}_\Gamma-\mathrm{div}_\Gamma(|\nabla_\Gamma\eta|^{p-2}\nabla_\Gamma\eta)+\alpha|\eta|^{p-2}\eta = f \quad\text{on}\quad \Gamma_t, \quad t\in(0,T)
\end{align*}
with $\alpha>0$ is considered, but the restriction $\alpha\neq0$ can be removed if one obtains a result similar to Proposition \ref{P:Est_Ave}.

\section*{Acknowledgments}
The work of the author was supported by JSPS KAKENHI Grant Number 23K12993.
The author would like to thank the anonymous referee for valuable comments on this work including the suggestion for extension to a Leray--Lions type operator.

\section*{Data availability}
No data was used for the research described in the article.

\bibliographystyle{abbrv}
\bibliography{We_pLap_MoDo_Ref}

\end{document}